\documentclass{article}
\usepackage[a4paper, portrait, margin=1.1811in]{geometry}
\usepackage[english]{babel}
\usepackage[utf8]{inputenc}
\usepackage[T1]{fontenc}
\usepackage{amsfonts, amsmath, amsthm, amssymb, bbm, mathrsfs, esint, dsfont, mathenv, relsize, stmaryrd}
\usepackage{mathtools}
\usepackage[normalem]{ulem}
\usepackage{enumerate, enumitem,multicol}
\usepackage{algorithm}
\usepackage{algpseudocode}
\usepackage{helvet}
\usepackage{etoolbox}
\usepackage{graphicx}
\usepackage{titlesec}
\usepackage{caption}
\usepackage{booktabs}
\usepackage[dvipsnames]{xcolor} 
\usepackage{csquotes}

\usepackage{setspace}
\graphicspath{ {./images/} }
\usepackage{scrextend}
\usepackage{fancyhdr}
\usepackage{graphicx}

\usepackage[backend=bibtex, natbib=true, style=authoryear, giveninits=true, maxbibnames=99, doi=false,isbn=false,url=false,eprint=false,uniquename=init]{biblatex}
\RequirePackage[colorlinks=true,citecolor=blue!50!black,linkcolor=blue!50!black,urlcolor=blue!50!black,pdfauthor=author]{hyperref}
\usepackage[capitalize,nameinlink,noabbrev]{cleveref}

\theoremstyle{plain}
\newtheorem{theorem}{Theorem}[section] 
\newtheorem{proposition}[theorem]{Proposition}

\newtheorem{corollary}[theorem]{Corollary}
\newtheorem{lemma}[theorem]{Lemma}
\theoremstyle{remark}

\newtheorem*{rem}{Remark} 

\crefname{thm}{Theorem}{Theorems}
\crefname{lem}{Lemma}{Lemmas}
\crefname{prop}{Proposition}{Propositions}
\crefname{cor}{Corollary}{Corollaries}

\def\EE{\mathbb{E}}

\def\PP{\mathbb{P}}

\def\RR{\mathbb{R}}

\def\1{\mathds{1}}
\def\B{\mathcal{B}}
\def\G{\mathcal{G}}
\def\I{\mathcal{I}}
\def\J{\mathcal{J}}
\def\M{\mathcal{M}}
\def\N{\mathcal{N}}
\def\S{\mathcal{S}}

\newcommand{\Tr}{\mbox{Tr}}

\DeclareMathOperator*{\argmin}{argmin}

\renewcommand{\hat}{\widehat}
\renewcommand{\tilde}{\widetilde}
\renewcommand{\bar}{\overline }

\newlist{hypDiamond}{enumerate}{1} 
\setlist[hypDiamond,1]{label={$\diamond$}}

\makeatletter   
\patchcmd{\@maketitle}{\LARGE \@title}{\fontsize{16}{19.2}\selectfont\@title}{}{}
\makeatother

\makeatletter
\def\namedlabel#1#2{\begingroup
   \def\@currentlabel{#2}%
   \label{#1}\endgroup
}
\makeatother

\usepackage{authblk}

\newsavebox\affbox
\author[1]{\textbf{Luca Castelli}}
\author[2]{\textbf{Irène Gannaz}}
\author[1]{\textbf{Clément Marteau}}

\affil[1]{Univ Lyon, Université Claude Bernard Lyon 1, CNRS UMR 5208, Institut Camille Jordan, F-69622 Villeurbanne, France\protect\\~ } 
\affil[2]{Univ. Grenoble Alpes, CNRS, Grenoble INP\footnote{Institute of Engineering Univ. Grenoble Alpes}, G-SCOP, 38000 Grenoble, France}

\titleformat{\section}{\normalfont\fontsize{10}{15}\bfseries}{\thesection.}{1em}{}
\titleformat{\subsection}{\normalfont\fontsize{10}{15}\bfseries}{\thesubsection.}{1em}{}
\titleformat{\subsubsection}{\normalfont\fontsize{10}{15}\bfseries}{\thesubsubsection.}{1em}{}

\titleformat{\author}{\normalfont\fontsize{10}{15}\bfseries}{\thesection}{1em}{}

\title{\textbf{\huge A non-asymptotic analysis of the single component PLS regression}\\}
\date{\today}    

\begin{document}

\pagestyle{headings}	
\newpage
\setcounter{page}{1}
\renewcommand{\thepage}{\arabic{page}}

\setlength{\parskip}{12pt} 
\setlength{\parindent}{0pt} 
\onehalfspacing  
	
\maketitle
	
\noindent\rule{15cm}{0.5pt}
	\begin{abstract}
	This paper investigates some theoretical properties of the Partial Least Square (PLS) method. We focus our attention on the single component case, that provides a useful framework to understand the underlying mechanism. We provide a non-asymptotic upper bound on the quadratic loss in prediction with high probability in a high dimensional regression context. The bound is attained thanks to a preliminary test on the first PLS component. In a second time, we extend these results to the sparse partial least squares (sPLS) approach. In particular, we exhibit upper bounds similar to those obtained with the lasso algorithm, up to an additional restricted eigenvalue constraint on the design matrix. 
	\end{abstract}
\noindent\rule{15cm}{0.4pt}


\section{Introduction}


We are interested in the classical linear model in a high dimensional context. We observe a $n$-sample $(X_i,Y_i)$, $i=1,\dots,n$, where the $Y_i\in\RR$ are outcome variables of interest and the $X_i\in\RR^p$ $p$-dimensional covariates. We consider a linear relationship within each couple $(X_i,Y_i)$, represented by the equation
\begin{equation}
\label{Eq:Modele lineaire}
Y=X\beta +\varepsilon,
\end{equation}
where $\varepsilon=(\varepsilon_1,\dots \varepsilon_n)^T\sim\N_n\big(0,\tau^2I_n\big)$, $X=(X_1,\dots, X_n)^T\in \RR^{n\times p}$ and $Y=(Y_1,\dots,Y_n)^T\in \RR^n$. The matrix $I_n$ is the identity matrix of size $n$ and the parameter $\tau$ characterizes the noise level. The exponent $T$ denotes the transpose operator. In this context, one might be alternatively interested in providing inference on the parameter $\beta$ itself, or on $X\beta$ (prediction task). The regression model~\eqref{Eq:Modele lineaire} has a long history. Several issues may arise, in particular in a high dimensional context, namely when $p$ is of the same order, or much larger, than the number $n$ of available observations. We refer to \citet{Giraud} for a comprehensive introduction to this topic. 

In this paper, we will focus our attention on the Partial Least Squares (PLS) principle. PLS was mainly developed in the chemometrics community \citep{Martens}. This approach has shown its ability for the prediction of regression models with many predictor variables \citep{Garthwaite1994}. It has been widely used in chemometrics \citep{SWOLD1995,SWOLD2001} but also in other fields such as social science \citep{sawatsky2015partial}, and biology \citep{palermo2009performance,yang2017application}. Several extensions have been proposed over the years as, e.g., \citet{delaigle2012methodology} for functional data or \citet{naik2000partial} for the single-index models.

The idea of PLS is to seek a fixed number of directions -- say $K \in \lbrace 1,\dots, p \rbrace$ -- formed by linear combinations of $X$ coordinates, which are highly correlated with the target variable $Y$ (see \citet{Aparicio} for a comprehensive introduction). These $K$ directions are gathered in a weight matrix $W \in\RR^{p\times K}$. The parameter $\beta$ is then estimated by an appropriate linear combination of the columns of $W$. More formally, the PLS estimator satisfies
 \begin{equation}
 \label{eq:optim}
     \hat{\beta}_W=\underset{w\in [W]}{\mathrm{argmin}} \ \|Y-Xw\|^2,
 \end{equation}
where $[W] \subset \mathbb{R}^p$ denotes the space spanned by the columns of the weight matrix $W$, and $\lVert.\rVert$ is the $\ell^2$-norm on $\RR^n$. 
The objective of the dimension reduction given by $W$ is to decrease the number of features from $p$ to $K$ while retaining as much information as possible.  On the contrary to Principal Components Reduction where the directions are built only considering the covariates $X$, PLS regression builds the weights $W$ iteratively, considering the successive correlations with the outcome $Y$, to increase the prediction quality. We refer to  \citet{FrankFriedman,KramerSugiyama} and to  \cref{s:overview} for more details regarding the way in which $W$ is constructed. 

Although the PLS principle has attracted a lot of attention over the years, few theoretical results have been obtained. Among others, we can mention \citet{Helland} where the space $[W]$ resulting from the PLS approach has been characterised, or \citet{cook2010envelope,Cook2013envelopes} for a connection between PLS regression and envelopes. While \citet{Chun} proved inconsistency of the PLS estimator when the number of covariates is too large, \citet{Cook,Cook2019} established -- up to strong constraints on $\beta$ and the design matrix $X$ -- the asymptotic behavior of the mean squared error of prediction and prove that it may tend to $0$ as the number of observations goes to infinity. Additional investigations have also proposed to take into account sparsity constraints in the PLS algorithm. We refer for instance to \citet{Durif, AlSouki}.

The aim of this paper is to investigate the theoretical performances of the PLS algorithm. We will focus our attention on the single component PLS algorithm, namely $K=1$. Despite its apparent simplicity, this setting provides numerous statistical challenges. In particular, the non-linearity of the corresponding estimator requires a careful attention. In this context, our aim is twofold: 
\begin{itemize}
    \item Obtaining non-asymptotic prediction bounds on the PLS estimator with high probability. In particular, we want to provide an extensive description of the performances and limitations of this approach with a minimal set of assumptions. 
    \item Extend these results to a sparse scenario, and discuss existing similarities with alternative methods like the lasso estimator (see, e.g., \citet{Tibshirani} and \cref{s:sparse_def}).
\end{itemize}

To this end, we establish in a first time a non-asymptotic bound on the prediction loss. Denoting by $\hat \beta_{PLS}$ be the PLS estimator with $K=1$, we prove in particular that with high probability, \[\frac{1}{n}\|X\hat{\beta}_{PLS}-X\beta\|^2\le B(\beta)+C\frac{\tau^2}{n} \max\left(\frac{\Tr(\Sigma)}{\lambda},\frac{\rho(\Sigma)\Tr(\Sigma)}{\lambda^2}\right),\] 
where $C$ is a positive constant, $\Tr(.)$ is the trace operator on $\RR^{p\times p}$, and $\Sigma = \frac{1}{n}X^T X\in \RR^{p\times p}$ is the Gram matrix associated to the design $X$. The term $\lambda$ corresponds to the norm of the first theoretical PLS component while $B(\beta)$ is a measure on the bias induced by the algorithm. The variance term can be read as a signal-to-noise ratio and allows to describe scenarios where the PLS estimator provides satisfying results. The formal result, along with an extended discussion, is displayed in \cref{s:contrib1}. 

Next, we extend this result to the parsimonious case using a sparse version $\hat\beta_{sPLS}$ of the algorithm, including an $\ell_1$ constraint in the optimisation process. Assuming that the Gram matrix $\Sigma$ satisfies a restricted eigenvalue condition, we establish that, with high probability,
\[\frac{1}{n}\|X\tilde{\beta}_{sPLS}-X\beta\|^2\le B(\beta) +C\frac{\tau^2s}{n}\ln(p),\]
where $s$ denotes the number of non-zero coefficient of the first PLS axis. In particular we recover, up to the bias term, the same kind of bound as those obtained for the Lasso procedure. See \cref{s:sparse} for more details on our results and related consequences.

The contribution is structured as follows. \cref{s:overview} provides a brief overview on the PLS algorithm and on its extensions. We detail in particular the case of a projection on a single component, that is, $K=1$. We also present the sparse algorithm studied in this paper. \cref{s:contrib1} gathers our first theoretical bound on the single PLS estimator and a discussion on some specific scenarios that may shed some light on the behaviour of the PLS principle. An extension to the sparse case is presented in \cref{s:sparse}, together with a corollary that involve an additional assumption on the Gram matrix $\Sigma$. The proofs of our main results are displayed in the Appendix.

All along the paper, we use the following notation. The $\ell^2$ (resp. $\ell^1$) norm in $\mathbb{R}^p$ is written $\|.\|_2$ (resp. $\|.\|_1$), while $\|.\|$ corresponds to the $\ell^2$ norm in $\mathbb{R}^n$. For any matrix $A$, $A^T$ corresponds to the transpose matrix, and $[A]$ to the vectorial space spanned by the columns of~$A$. If $A$ is a square matrix, we write $\Tr(A)$ (resp. $\rho(A)$) for the trace (resp. spectral radius) of $A$. Here and below, the design matrix $X$ is considered as deterministic. We denote by $\Sigma = X^TX/n$ the Gram matrix associated to the design $X$, $\hat \sigma = X^TY/n$ the empirical covariance between $X$ and the target $Y$ and $\sigma = \mathbb{E}[\hat\sigma] = \Sigma\beta$ its population counterpart.


\section{A short PLS overview}
\label{s:overview}
 
In this section, we first recall the PLS algorithm. We will focus our attention on the single component case, and will briefly discuss existing results in this particular setting. Next, we present a sparse extension of this approach that includes a lasso type penalization in the optimisation process. 
 
\subsection{The PLS algorithm}

For a fixed number of components $K\in \lbrace 1,\dots, p \rbrace$, the PLS algorithm summarizes some directions~$w_k$ for $k\in \lbrace 1,\dots, K \rbrace$ whose associated components $t_k$ are the most possible correlated with the target variable $Y$. The $K$ PLS components $(t_k)_{k=1,\dots,K}$ form an orthogonal basis of $\mathbb{R}^p$. They are build to compress the data in uncorrelated terms at each iterative step. The corresponding weights $W=(w_k)_{k=1,\dots,K}$ are computed iteratively, seeking the components which are the most correlated with the predicted variable $Y$ on residuals.

For each new iteration, the algorithm advances by executing a deflation step: it takes the influence of the previously computed components out of the design matrix $X$. Let $X^{(1)}=X$ and for $k \in \lbrace 2,\dots, K \rbrace$, define
$$ X^{(k)}=X-P_{[t_1,t_2,\dots,t_{k-1}]}(X)=X^{(k-1)}-P_{[t_{k-1}]}(X^{(k-1)}),$$
where $P_{[t_1,t_2,\dots,t_k]}$ denotes the orthogonal projection on the subspace spanned by $t_1,t_2,\dots,t_k$. The weight $w_k$ is then solution of the following optimisation problem 
\begin{equation}\label{Eq:Opti:PLS}
    w_k=\argmin_{w\in\RR^p}(-Y^T X^{(k)} w)\quad \mathrm{s.t.}\quad \|w\|_2=1.
\end{equation}
Components $t_k$ are defined as $t_k=X^{(k)}w_k$, for all $k=1,\ldots,K$.

The formal PLS algorithm is displayed in \cref{Algo:PLS} below. 

\begin{algorithm}
\caption{PLS Algorithm}\label{Algo:PLS}
Input \textbf{X},Y and K
\begin{algorithmic}
\State $\textbf{X}_1\textbf{=}\textbf{X}$
\For{k\textbf{=}1,\dots, K}
        \State $\textbf{w}_k\textbf{=}\textbf{X}^{(k)^T} Y/ \| \textbf{X}^{(k)^T} Y\|_2$\quad (loadings computation)
        \State $\textbf{t}_k\textbf{=}\textbf{X}^{(k)}\textbf{w}_k$ \quad (component construction)
        \State $\textbf{X}^{(k+1)}\textbf{=}\textbf{X}^{(k)}\ \textbf{-}\ \textbf{P}_{[\textbf{t}_k]}(\textbf{X}^{(k)})$ \quad(deflation step)
\EndFor
\end{algorithmic}
\end{algorithm}

This algorithm is described for example in \citet{Helland} where some of its properties are discussed. Alternative versions of this algorithm have been proposed over the years. For example, the components $t_k$, $k\in \lbrace 1,\dots, K\rbrace$ can be calculated using variants as the NIPALS (non linear iterative least squares) introduced in \citet{WOLD1975} or SIMPLS (straightforward implementation of a statistically inspired modification of the PLS method) proposed by \citet{DEJONG1993} which computes linear combinations of the original variables. We refer to \citet{algos} for an overview of PLS algorithms, as well as the equivalence with the conjugate gradient method.

The PLS components $(t_1,\dots, t_K)$ can be gathered in a matrix $T \in \mathbb{R}^{n\times K}$, where each column $k$ of $T$ corresponds to the component $t_k$. In particular, it can be noticed that $[T] = [XW]$ where $W = (w_1,\dots, w_K) \in \mathbb{R}^{p\times K}$ is the weight matrix. The PLS prediction $\hat{Y}_{PLS}:= X\hat\beta_{PLS}$ is given by 
$$\hat{Y}_{PLS}=T(T^T T)^{-1}T^T Y=P_{[XW]}(Y).$$
It immediately follows that the corresponding estimator of $\beta$ is computed as 
\begin{equation}
\label{eq:betaPLS}
\hat{\beta}_{PLS}=\hat{\beta}_{W}=W(W^T \Sigma W)^{-1}W^T \hat{\sigma},
\end{equation}
where $\hat\sigma = X^TY/n$ denotes the empirical covariance vector between $X$ and $Y$. 

The iterative nature of the PLS algorithm makes the corresponding estimation problem difficult to handle. However, \citet{Helland} proved that $[W]=\hat{ \G}$, where $\hat\G$ denotes the Krylov space which is the space generated by $\lbrace \hat\sigma, \Sigma\hat\sigma, \dots, \Sigma^{K-1} \hat\sigma \rbrace$. In particular, the optimisation problem \eqref{eq:optim} can be rewritten as 
\begin{equation}
\label{eq:betaKrylov}
    \hat\beta_{PLS} \in \argmin_{w\in \hat\G} \| Y-Xw\|^2.
\end{equation}
The latter provides an alternative expression which is sometimes useful from a mathematical point of view (see for instance \cite{Cook2019}). Nevertheless, whatever the optimisation form, the resulting PLS estimator is always non-linear in the target $Y$. This induces several issues for obtaining accurate prediction bounds. In this context, we will focus in this paper on the single component  case $(K=1)$. This case is a starting point preceding further additional investigations.

\subsection{The single component PLS}

In the specific case where the number of component $K$ is constrained to be equal to $1$, an explicit and workable expression of $\hat\beta_{PLS}$ can be obtained. Indeed, a direct application of \eqref{Eq:Opti:PLS}, \eqref{eq:betaPLS} or \eqref{eq:betaKrylov} leads to
\begin{equation}
\label{eqn:beta}
\hat\beta_{PLS}=\frac{\hat \sigma^T \hat \sigma}{\hat \sigma^T  \Sigma\hat \sigma} \hat \sigma.
\end{equation}
In particular, it can be noticed that the estimator $\hat\beta_{PLS}$ is handled by the empirical covariance $\hat\sigma$ which corresponds to the first PLS direction $w_1$ up to a normalizing constant. 

The formulation \eqref{eqn:beta} has been at the core of the investigations presented in \citet{Cook}. The latter established, up to our knowledge, the first consistency results for this estimator in several scenarios, including in particular high dimensional frameworks. Nevertheless, a central assumption made by the authors is the fact that the so-called population version of the Krylov space is of dimension $1$. More formally, writing $\bar\sigma = \bar\Sigma \beta$ and $\bar\Sigma = \mathbb{E}(\Sigma)$ (assuming that the covariates are random variables), they suppose that $\dim\overline{\mathcal{G}} =1$ where $\overline{\mathcal{G}}$ is the vectorial space generated by $\lbrace \bar\sigma, \bar\Sigma\bar\sigma, \dots, \bar \Sigma^{p-1}\bar\sigma\rbrace$. An immediate consequence of this assumption is that the coefficient vector $\beta$ is strongly constrained. Indeed
$$ \dim(\overline{\mathcal{G}}) =1 \quad \Rightarrow \quad \beta \in [\bar\sigma],$$
provided $\bar \Sigma$ is invertible. Moreover, $\dim(\overline{\mathcal{G}}) =1$ implies that $\bar\sigma$ is an eigenvector of $\overline{\Sigma}$.

In \cref{s:contrib1}, we propose a bound on the prediction error with a controlled probability that does not use the assumption $\dim(\overline{\mathcal{G}})=1$. Moreover, our bound is non-asymptotic and does not require a specific regime for $n$ and $p$. This non-asymptotic analysis shed lights on some particular scenarios, missing from \citet{Cook}, where the term $\sigma^T \Sigma\sigma$ is too small to ensure a precise control of the denominator in \eqref{eqn:beta}.

\subsection{Sparse-PLS}
\label{s:sparse_def}
In presence of a high number of regressors, one might want to promote dimension reduction during the estimation process. Following \citet{Tibshirani}, a possible strategy is to keep the more relevant variables and to shrink the others to zero by adding a $\ell^1$-penalty on the coefficients in the objective function. Using this principle, the Sparse Partial Least Squares (sPLS) algorithm uses a sparse constraint to select a subset of variables that have the highest correlation with the response variable in the construction of the components. It is is an effective tool for variable selection, dimension reduction, and developing predictive models using a limited set of relevant predictors. The sPLS estimator has been considered in several frameworks and application fields as, e.g., \citet{lee2011sparse,abdel2014comparison} in chemometrics, \citet{LeCao, Chun} in genomics or \citet{fuentes2015sparse} in ecometrics. 

Following \citet{Durif}, a sparsity constraint is included in the estimation process by adding a $\ell^1$-penalty in the optimisation problem \eqref{Eq:Opti:PLS}. More formally, the latter is replaced by 
\begin{equation}
\label{Eq:Opti:sPLS}
    \tilde{w}_k=\argmin_{w\in \RR^p}\left[-\frac{1}{n}Y^T X^{(k)}w+\mu\|w\|_1\right]\quad \mathrm{s.t.}\quad \|w\|_2=1,
\end{equation}
where $\mu>0$ denotes a regularization parameter that determines the sparsity level of the solution and where, for $k\in \lbrace 2,\dots K\rbrace$, $X^{(k)}$ is a deflated version of $X$ on the previous components $\tilde t_1,\dots,\tilde t_{k-1},$ with $\tilde t_j=X^{(j)}\tilde w_j$, $j\leq k-1$. Several alternative approaches have been proposed over the last two decades to try to force sparsity of the solution. We refer, e.g. to \citet{LeCao} and \citet{Chun}. 

\citet{Durif} provide a closed-form solution for the optimisation problem using proximal operator. This expression can be made even simpler in the specific case where $K=1$ which is at the core of this paper. In particular, the corresponding estimator $\hat\beta_{sPLS}$ involves a soft-thresholded version of $\hat\sigma$ (see \citet{donoho1994shrinkage} and \cref{Prop:sPLS}). 

In \cref{s:sparse}, we provide a complete non-asymptotic analysis of the prediction error associated with the estimator $\hat\beta_{sPLS}$, still in the single component case. We exhibit and discuss several scenarios where this algorithm provides accurate predictions in a high dimensional context. These scenarios include in particular the specific cases where the Gram matrix $\Sigma$ satisfies a restricted eigenvalue condition (see for instance \cite{Tsyb}).


\section{Theoretical results with single component PLS}
\label{s:contrib1}

Our first contribution is a non asymptotic control of the prediction loss, with an explicit upper bound. We also take into account in our procedure the reliability of the PLS algorithm. That is, we identify in the estimation scheme when the PLS regression does not bring enough information to ensure a good quality of prediction. Our estimator is modified accordingly.

\subsection{Non-asymptotic control}

We can notice from \eqref{eqn:beta} that the single component PLS estimator is made of a direction $\hat\sigma$ and a so-called intensity $\hat\sigma^T\hat\sigma/\hat\sigma^T\Sigma\hat\sigma$. A random term appears in the denominator of this last quantity. If the deterministic counterpart of this random part is too small, in a sense which is made precise in the proof, we cannot expect to have an accurate control of the variance term. To this end, we slightly modify the PLS estimator and introduce a threshold for the denominator.  Namely $\hat\beta_{PLS}$ is considered for the estimation if and only if $\hat\sigma^T \Sigma \hat\sigma$ is large enough compared to a minimal reference value $p_n$ (see \cref{Th:M1} for more detail). We stress that 
$$ \hat\sigma^T \Sigma \hat\sigma = \frac{1}{n} \lVert X\hat\sigma\rVert = \frac{1}{n} \lVert \hat t_1 \rVert^2,$$
where $\hat t_1$ denotes, up to a normalizing constant, the first PLS component. 
Intuitively, if the norm of the first component is close to zero, the PLS estimation will not be accurate and the estimator is replaced by $0$. On the other hand, if its value is above a fixed level of inertia, we use the PLS regression algorithm. Applying this principle, which can be considered as a pre-processing testing procedure, we obtain the following theorem.

\begin{theorem}
\label[thm]{Th:M1} 
Let $\delta \in (0,1)$. 
Define
\[\hat{\beta}_\delta=
\begin{cases}
\hat{\beta}_{PLS} &  \ \text{if  $\hat\sigma^T \Sigma\hat\sigma > \mathrm{t}_{\delta}\,p_n$}
\\ 
0 & \text{otherwise}
\end{cases} \quad \mathrm{with} \ \quad \ p_n=\frac{\tau^2}{n}\rho(\Sigma)\Tr(\Sigma), \]
for some explicit threshold constant $\mathrm{t}_\delta$ depending only of $\delta$.
Then, with a probability higher than $1-\delta$, there exist a constant $C_\delta>0$, depending on $\delta$, such that
	\begin{equation}
	\label{eq:Bound1}
		\frac{1}{n}\|X\hat{\beta}_\delta-X\beta\|^2\le \frac{2}{n}\underset{v\in[\sigma]}{\inf}\|X(\beta-v)\|^2+C_{\delta}\frac{\tau^2}{n} \max\left(\frac{\Tr(\Sigma)}{\lambda},\frac{\rho(\Sigma)\Tr(\Sigma)}{\lambda^2}\right),
	\end{equation} 
where
\begin{equation}\label{eq:Lambda}
    \lambda = \frac{\sigma^T \Sigma \sigma}{\|\sigma\|_2^2}. 
\end{equation}
\end{theorem}

The proof of \cref{Th:M1} is postponed to \cref{s:proof1}. It is mainly based on non-asymptotic deviation results on non-centered weighted $\chi^2$ distributions. The proof manages the different values of the indicator function $\mathds{1}_{\lbrace \hat\sigma^T \Sigma \hat\sigma  \geq \mathrm{t}_\delta p_n\rbrace}$ where $\mathrm{t}_\delta$ is defined in \eqref{eq:cdelta}. In particular, we prove that the latter is equal to $1$ when the amount of signal inside the first (theoretical) PLS component $\sigma^T \Sigma \sigma$ is large enough. On the other hand, we get an upper bound on the prediction loss when $\hat{\beta}_\delta=0$. This situation corresponds to the case where the first component does not contain significant information. We highlight the fact that it implies that $X\beta$ is relatively low, compared to the level of noise. 
The analysis of the bound displayed in \cref{Th:M1} is discussed in the next subsection.

\subsection{Discussion}

The bound displayed in \cref{Th:M1} is composed of two terms. The first one, equal to $$\frac{2}{n}\underset{v\in[\sigma]}{\inf}\|X(\beta-v)\|^2,$$ 
represents a bias term. It measures the distance between the true signal $X\beta$ and the best possible prediction based on the first theoretical PLS axis $[\sigma]$. This term cancels out provided $\beta \in [\sigma]$. This situation occurs for instance in the case where both the dimension of the Krylov space $\mathcal{G}=\mathrm{Vect}\left\lbrace \sigma, \Sigma\sigma,\dots, \Sigma^{p-1}\sigma \right\rbrace$ is equal to $1$ and the matrix $\Sigma$ is invertible. This assumption was -- in a slight different version -- at the core of \citet{Cook}. 

The second term displayed in the right hand side of \eqref{eq:Bound1} can be considered as a variance term. It essentially measures the impact of the noise $\varepsilon$ on the PLS algorithm. It involves a ratio between the trace of $\Sigma$ and the term $\lambda$ introduced in \eqref{eq:Lambda}.  The latter exactly corresponds to the norm of the first theoretical PLS component $t_1:= \frac{1}{\sqrt{n}} Xw_1$ associated to the first normalized theoretical PLS axis $w_1 = \sigma/\|\sigma\|_2$. In other words, the term 
$$ 1 \leq  \frac{\Tr(\Sigma)}{\lambda} \leq +\infty,$$
can be considered as an inverse relative inertia. This interpretation makes sense in the particular case where $\mathrm{dim}(\mathcal{G}) = 1$. In such a case, it can be proved that the first PLS axis $w_1$ is an eigenvector of $\Sigma$ associated to the eigenvalue $\lambda$ (see \cite{Cook}). In the general case, namely when $\dim(\mathcal{G})$ is not necessarily equal to $1$, it provides in some sense an inverse of a signal-to-noise ratio that controls the accuracy of the single PLS component. If this ratio is close to $1$, the first PLS component captures most of the inertia of the data and we obtain a variance term with a parametric rate $\tau^2/n$. On the other hand, if this ratio is large, we cannot expect to get accurate results for the single component PLS: the amount of signal captured in the first component is not large enough to counterbalance the presence of the noise in the data. 

This discussion can be illustrated by considering two extreme examples. First, assume that $\Sigma = I_p$, the identity matrix in $\mathbb{R}^{p\times p}$. In this particular case, we get that $\Tr(\Sigma) = p$ and $\lambda = 1$. The variance term in \cref{Th:M1} is then of order $\tau^2 p/n$, which can be dramatically large in a high dimension setting. This appears to be quite natural since in such a situation, the explanatory variable are uncorrelated: nothing can be expected from the PLS algorithm which is designed to express the couple $(Y,X)$ in a low dimensional space. On the other hand, the case where $\Sigma$ has a rank equal to one with largest eigenvalue $\lambda$ is the most favorable case since in such a setting, $\Tr(\Sigma)/\lambda = 1$ and the variance term is equal to $\tau^2/n$. 

To conclude this discussion, we stress that our bound on the prediction loss is a generalization of the one displayed in \citet{Cook}. It nevertheless differs from this former result by several items. First, as discussed above, it does not impose an assumption on the dimension of the Krylov space $\mathcal{G}$ and on the rank of $X$. An immediate consequence is the presence of a bias term. Secondly, our bound is non-asymptotic: we provide a control with high probability, whatever the values of $n$ and $p$. Such an approach allows to consider all the possible values of $\|w_1\|_2$ and $\| t_1\|$: we do not require that these quantities are bounded from below. In particular, the fact the our variance term is slightly different from \citet{Cook} is an immediate consequence of this approach.


\section{Sparse estimation}
\label{s:sparse}

\subsection{Statistical performances of the single component s-PLS estimator}
\label{s:spls}

First, we focus our attention on the sparse version of the single component PLS estimator. In particular, we consider the estimator $\hat\beta_{sPLS}$ defined in \eqref{eq:betaPLS} with $W=\tilde w_1$ where $\tilde w_1$ is solution of the optimisation problem \eqref{Eq:Opti:sPLS}. Recall that the latter aims at providing a sparse approximation of the first theoretical PLS axis $w_1 = \sigma/\|\sigma\|_2$. We are hence implicitly interested in the situations where the number of non-zero coefficients of $\sigma \in \mathbb{R}^p$ is small compared to the number of available data $n$. In particular, we take advantage of the expression of $\hat\beta_{sPLS}$ which involves a shrinkage operator.

Denote by $\J_0=\{j=1,\dots,p, \sigma_j\ne0\}$ the support of $\sigma$. Hereafter, a vector $v_{\I}$ denotes the vector $v$ where all $v_j$ are set to 0 for $j\notin \I.$
The purpose of this section is to achieve bounds on the quadratic loss in prediction with terms depending only on the variables which contribute to the first component. Writing $\Sigma_{\J_0}=\frac{1}{n}X_{\J_0}^T X_{\J_0}$ with $(X_{\J_0})_{ij}=0$ if $j\notin \J_0$, $i=1,\dots,p$, we expect to reduce the variance appearing in \eqref{eq:Bound1} by exhibiting terms of order $\Tr(\Sigma_{\J_0})$ instead of $\Tr(\Sigma)$. Before presenting our results, we need two different assumptions on the model.

\begin{description}[font=\bf]
\item[Assumption A.1.]\namedlabel{ass:A1}{{Assumption A.1}} The columns of $X$ are normalized, namely
$$ \Sigma_{jj} = \frac{1}{n}\sum_{i=1}^p X_{ij}^2 = 1 \quad \forall j\in \lbrace 1,\dots, p \rbrace.$$
\end{description}
 
This assumption is quite standard for the linear model \eqref{Eq:Modele lineaire}. It allows to get simpler expression for the penalty level $\mu$ that appears in the optimisation process \eqref{Eq:Opti:sPLS}. This assumption can be guaranteed thanks to a scaling of the explanatory variables. 

\begin{description}[font=\bf]
\item[Assumption A.2.]\namedlabel{ass:A2}{{Assumption A.2}} Let $\delta \in (0,1)$ be fixed. There exists a term $d_{\delta,p}$ such that
$$ \sigma^T \Sigma\sigma > d_{\delta,p}\frac{\tau^2}{n}\rho(\Sigma_{\J_0})\Tr(\Sigma_{\J_0}).$$
\end{description}

We recall that, up to a normalizing constant, the quantity $\sigma^T\Sigma\sigma$ corresponds to the norm of the first theoretical PLS component. \ref{ass:A2} can hence be understood as minimal energy on the signal. In \cref{s:contrib1}, we have discussed the fact that situations where $\sigma^T\Sigma\sigma$ is small create numerous issues in the control of the prediction loss. These issues have been circumvented in \cref{Th:M1} thanks to the introduction of a threshold on the denominator. In a sparse context, constructing a test allowing to get rid of this assumption appears to be quite involved. It would require in particular a precise knowledge on the location of the support $\mathcal{J}_0$. This is not reasonable in practice. Nevertheless, we will present in \cref{s:alternative} below an alternative procedure that allow to remove \ref{ass:A2}, up to an additional constraint on the design matrix. 

A control for the prediction loss associated to the sPLS estimator $\hat\beta_{sPLS}$ is displayed below. The proof is postponed to \cref{s:proof2}.

\begin{theorem}
\label[thm]{Th:M2}
Let $\delta\in (0,1/2)$ be fixed. Assume that \ref{ass:A1} and \ref{ass:A2} are satisfied with $d_{\delta,p} = C_0\left(\ln\bigl(\frac{10}{\delta}\bigr)+\ln\bigl(\frac{p}{\delta}\bigr)\right)$, with $C_0>0$ an explicit constant. Let  
\begin{equation}
\label{eq:mu}
\mu  = 2\tau \sqrt{\frac{2}{n}\ln\left(\frac{2p}{\delta}\right)}.
\end{equation}
Then, we get, with probability greater than $1-\delta$,
	\begin{equation}
		\frac{1}{n}\|X(\hat\beta_{sPLS}-\beta)\|^2\le \frac{2}{n}\underset{v\in[\sigma]}{\inf}\|X(\beta-v)\|^2 +  D_{\delta}\frac{\tau^2 s}{n}\max\left(\frac{\rho(\Sigma_{\J_0})}{\lambda^2},\frac{1}{\lambda} \right) \ln\left(\frac{p}{\delta}\right),
	\end{equation}
where $s:= |\mathcal{J}_0| = \Tr(\Sigma_{\J_0})$ denotes the size of the support and $D_\delta$ a positive constant depending only on $\delta$.
\end{theorem}
The constant $C_0$ is not detailed here for the sake of clarity. It is given explicitly in the proof, in \cref{s:ddelta}.

The variance term in the bound obtained above differs from the one displayed in \cref{Th:M1} (although the bias remains the same). This an immediate consequence of the $\ell_1$ constraint introduced during the optimisation process. While the first PLS direction $\hat w_1$ is equal to $\hat\sigma$ (up to a normalizing constant) in the framework of \cref{s:contrib1}, here we have $\hat w_1 = \tilde\sigma$ where $\tilde\sigma$ is a thresholded version of $\hat\sigma$. In particular, the support of $\tilde\sigma$ is expected to be close (in a sense which is made precise in the proof) of $\mathcal{J}_0$. As a consequence, the variance associated to the estimation of the $\sigma$ is 
$$ \frac{\tau^2}{n} \Tr(\Sigma_{\mathcal{J}_0}) = \frac{\tau^2s}{n} \leq \frac{\tau^2}{n} \Tr(\Sigma), $$
where for the first equality, we have used \ref{ass:A1}. In the same spirit, we can notice that $\rho(\Sigma_{\J_0})\le \rho(\Sigma)$. In particular, we can expect a significant improvement of our bounds in the situation where the first theoretical PLS direction $w_1= \sigma$ is sparse, namely when $|\J_0|=s<<p$. The counterpart of this improvement is a term of order $\log(p)$ in the bound. Such a term is quite standard in the literature. 

As for the standard framework (\cref{s:contrib1}), the ratio between the amount of signal available in the first PLS component described by $\lambda$ and the noise term discussed above plays an important role in the behaviour of the method. We can get rid of this ratio by using a classical assumption when dealing with sparsity contraints.  

\begin{description}[font=\bf]
\item[Assumption A.3 (Restricted eigenvalue condition)]\namedlabel{ass:A3}{{Assumption A.3}}
There exists a constant $\phi>0$ such that,
$$ \min_{\gamma: \|\gamma_{\J_0^C}\|_1 \leq 3 \|\gamma_{\J_0}\|_1} \frac{1}{n}\frac{ \| X\gamma \|^2}{\|\gamma \|_2^2} \geq \frac{1}{\phi}.$$
\end{description}

This kind of assumption has been at the core of several contributions and discussions. It requires that for any $\gamma \in \mathbb{R}^p$, the norm of $\|X\gamma\|/\sqrt{n}$ is comparable to $\|\gamma\|_2$ provided the signal in $\gamma$ in mainly concentrated in $\mathcal{J}_0$. We refer to \citet{Tsyb} among others, where this assumption is discussed and compared to other types of constraints. In the following, we apply this assumption on some vectors $\gamma\in \mathbb{R}^p$ whose support is included in $\J_0$. Accordingly, \ref{ass:A3} could be weakened by just assuming that the restricted matrix $X_{\J_0}$ is full rank. The following result is an almost immediate consequence of \cref{Th:M2} that takes advantage of \ref{ass:A3}. 

\begin{corollary}
\label[cor]{Cor:M2}
Let $\delta\in (0,1/2)$ be fixed. Assume that \ref{ass:A1}, \ref{ass:A2} and \ref{ass:A3} are satisfied. Setting 
$$ \mu = 2\tau \sqrt{\frac{2}{n}\ln\left(\frac{2p}{\delta}\right)},$$
we get, with probability greater than $1-\delta$,
	\begin{equation}
		\frac{1}{n}\|X(\hat\beta_{sPLS}-\beta)\|^2\le \frac{2}{n}\underset{v\in[\sigma]}{\inf}\|X(\beta-v)\|^2 +  D_{\delta}\frac{\tau^2 s}{n}\max\left(\rho(\Sigma_{\J_0}),1\right) \ln\left(\frac{p}{\delta}\right),
	\end{equation}
for some positive constant $D_{\delta}$.
\end{corollary}

The proof is postponed to \cref{s:proof_cor}. \ref{ass:A3} allows to lower bound the quantity $\lambda$. Up to the spectral radius $\rho(\Sigma_{\mathcal{J}_0})$ (which can reasonably be assumed to be bounded), we hence obtain a variance term 
$$ D_\delta \frac{\tau^2}{n} \ln\left( \frac{p}{\delta}\right),$$
which is exactly of same order as those obtained for the Lasso estimator and its variants. We refer, e.g., to \citet{Tsyb}, \citet{M_13} or \citet{DHL_17}.

\subsection{An alternative procedure}
\label{s:alternative}

We recall that the results displayed in \cref{Th:M2} heavily rely on \ref{ass:A2} which in some sense requires a minimal energy in the first theoretical PLS component. To get round of this issue, we slightly modify the single component sparse PLS estimator $\hat\beta_{sPLS}$. Recall the latter is defined in \eqref{eq:betaPLS} with $W=\tilde w_1$ and where $\tilde w_1$ is solution of the optimisation problem \eqref{Eq:Opti:sPLS}. Using simple algebra, we can in particular establish that \begin{equation} 
\hat \beta_{sPLS} = (\tilde w_1^T \Sigma \tilde w_1)^{-1} \tilde w_1^T \hat\sigma \times \tilde w_1,
\label{eq:bsparse}
\end{equation}
(see \cref{Prop:sPLS}).
Under \ref{ass:A2}, the inverse of $\tilde w_1^T \Sigma \tilde w_1$ exists with high probability. We propose here to replace $\hat \sigma$ in \eqref{eq:bsparse} by a thresholded version $\tilde \sigma$, \[\tilde\sigma_j = \mathrm{sgn}(\hat\sigma_j)(|\hat\sigma_j| - \mu)_+ \quad \forall j\in\lbrace 1,\dots, p\rbrace,\] where for any $x\in\mathbb{R}$, $\mathrm{sgn}(x) = 1$ if $x>0$, $-1$ if $x<0$, and 0 if $x=0$, and where $x_+ = x$ if ${ x\geq 0}$ and 0 otherwise. This heuristic leads to the the following estimator
\begin{equation} 
\tilde \beta = (\tilde w_1^T \Sigma \tilde w_1)^{-1} \tilde w_1^T \tilde\sigma \times \tilde w_1:= \tilde \lambda^{-1} \tilde w_1.
\label{eq:bsparseh}
\end{equation}
By taking advantage of the restricted eigenvalue condition (\ref{ass:A3}), we are able to manage the values of $\tilde\lambda$ without using \ref{ass:A2}. This discussion is formalized in the following theorem. 

\begin{theorem}
\label[thm]{Th:M3}
Let $\delta\in (0,1/2)$ be fixed and $\tilde\beta$ the estimator introduced in \eqref{eq:bsparse}. Assume that \ref{ass:A1} and \ref{ass:A3} are satisfied. Setting 
\begin{equation*}
\mu  = 2\tau \sqrt{\frac{2}{n}\ln\left(\frac{p}{\delta}\right)},
\end{equation*}
we get, with probability greater than $1-\delta$,
	\begin{equation}
		\frac{1}{n}\|X(\tilde{\beta}-\beta)\|^2\le \frac{2}{n}\underset{v\in[\sigma]}{\inf}\|X(\beta-v)\|^2 +  D_{\delta}'\frac{\tau^2 s}{n}\max\left(\rho(\Sigma_{\J_0}),1 \right) \ln\left(\frac{p}{\delta}\right),
	\end{equation}
where $s:= |\mathcal{J}_0| = \Tr(\Sigma_{\J_0})$ denotes the size of the support and $D'_\delta$ a positive constant depending only on $\delta$.
\end{theorem}

The proof is displayed in \cref{s:proof3}. The bound displayed in \cref{Th:M3} can be compared to the one presented in \cref{Cor:M2}. We obtain, up to some constants, exactly the same bounds, but without requiring \ref{ass:A2}.

\section{Conclusion}

In this contribution, we have provided non-asymptotic bounds on the prediction loss for the single component PLS estimator, including a sparse version. Our aim was to introduce minimal assumptions on the model allowing a control on the prediction loss. This allows in particular to shed light on several scenarios where the PLS estimator can lead to interesting and pertinent predictions.  

The different contributions displayed all along our paper (Theorems \ref{Th:M1}, \ref{Th:M2} and \ref{Th:M3} in particular) entail that the PLS algorithm does not immediately lead to satisfying theoretical results. Indeed, we need either strong additional constraints on the signal \eqref{ass:A2} or the introduction of a pre-processing step in the algorithm: a test (\cref{Th:M1}) or a soft-thresholding rule for $\hat\sigma$ (\cref{Th:M3}).

Our investigations are limited to the single component PLS estimator. We can indeed notice that this setting already brings a lot of technical issues. Nevertheless, it is expected to propose a similar non-asymptotic treatment in the general case with $K$ PLS components for some $K\in \lbrace 1,\dots,p \rbrace$. This might be the core of a future contribution.

\section*{Acknowledgements}
The authors would like to dedicate this contribution to François Wahl who suddenly passed away in April 2022. He initiated discussions on the PLS algorithm some years ago that created new research opportunities in our team.


\appendix

\section{Technical results}
\label{s:technical}
This section is dedicated to some specific technical results that will be used all along the proofs.
\subsection[Main distributions]{Distributions of $\hat\sigma$ and $\Sigma^{\frac{1}{2}}\hat\sigma$}

We first state the moments and the distribution of the main quantities appearing in the construction of the single component PLS estimator. 

\begin{lemma}\label{lem:variances}
We have
$$ \hat\sigma \sim\N_p\Bigl(\sigma,\frac{\tau^2}{n}\Sigma\Bigr) \quad \mathrm{and} \quad \Sigma^{\frac{1}{2}}\hat\sigma \sim\N_p\Bigl(\Sigma^{\frac{1}{2}}\sigma,\frac{\tau^2}{n}\Sigma^{2}\Bigr).$$
In particular
$$  \EE[\hat{\sigma}^T\hat{\sigma}] =\sigma^T\sigma+\frac{\tau^2}{n}\Tr(\Sigma),  \quad \quad  \EE[\hat{\sigma}^T\Sigma\hat{\sigma}] =\sigma^T\Sigma\sigma+\frac{\tau^2}{n}\Tr(\Sigma^2),$$
and 
$$ \EE[(\hat{\sigma}-\sigma)^T\Sigma(\hat{\sigma}-\sigma)] =\frac{\tau^2}{n}\Tr(\Sigma^2).$$
\end{lemma}

The results of this lemma are a direct consequence of Model \eqref{Eq:Modele lineaire} and of the fact that $\varepsilon \sim \mathcal{N}(0,\tau^2 I_n)$. The proof is thus omitted.

\subsection[Deviation inequalities]{Deviation inequalities on $\hat\sigma$ and $\Sigma^{1/2}\hat\sigma$}

The following proposition is an extension of a result established in \citet{Laurent}. It provides deviation inequalities for some quadratic function of a Gaussian vector.

\begin{proposition}\label[lem]{Lemme:ExtensionLaurent}
Let $U\sim\N_D(m,tA)$ with $D\in \mathbb{N}$, $m \in \mathbb{R}^D$, $t\in \RR_+$ and $A\in \mathbb{R}^{D\times D}$ a symmetric positive matrix. Define, for any $s\in \mathbb{N}$, 
$$\Theta_s=t^2\Tr(A^{2(s+1)}) +2t\rho(A^{s+1})\| A^{\frac{s}{2}}m\|^2,$$ 
Then, for all $s\in \mathbb{N}$ and $x\geq 0$,
\begin{align*}
    i) & \quad \PP\bigg(U^TA^sU-\EE[U^TA^sU]\ge 2\sqrt{\Theta_s x}+2t\rho(A)^{s+1}x\bigg) \le e^{-x},\\
    ii) &  \quad \PP\bigg(U^TA^sU-\EE[U^TA^sU]\le -2\sqrt{\Theta_s x}\bigg) \le e^{-x}.
\end{align*}
\end{proposition}
\begin{proof}
Denote by $(\lambda_i)_{j=1}^{D}$ the eigenvalues of $A$ and $\Lambda=\mathrm{diag}(\lambda_1,\dots, \lambda_D)$. Set $\theta=P^TA^{\frac{s}{2}}m$, where $P$ is a matrix verifying $A=P\Lambda P^T$. 
First remark that 
    $$U^TA^sU=U^TA^{\frac{s}{2}}PP^TA^{\frac{s}{2}}U=\|P^TA^{\frac{s}{2}}U\|^2,$$
Moreover, $P^TA^{\frac{s}{2}}U\sim\N(\theta,t\Lambda^{s+1})$ where $\Lambda^{s+1}$ is a diagonal matrix. The result then follows from a direct application of Lemma 2 from \citet{Laurent} on the Gaussian vector $P^TA^{\frac{s}{2}}U$ and from the bound 
$$ \sum_{j=1}^D \lambda_j^{s+1} \theta_j^2 \leq \max_{j=1\dots D} \lambda_j^{s+1} \times \|\theta\|^2 = \rho(A^{s+1})\| A^{\frac{s}{2}}m\|^2.$$
\end{proof}

Before stating additional results, we introduce, for any $x\in \mathbb{R}^+$, the following quantities:
\begin{align}
\label{eqn:T1}	\mathbf{T}_1(x)&=g(x)\frac{\tau^2}{n}\Tr(\Sigma)+2\sqrt{2}\sqrt{\frac{\tau^2}{n}}\rho(\Sigma)^{\frac{1}{2}}\sqrt{x}\|\sigma\|_2,\\
\label{eqn:T2}	 \mathbf{T}_2(x)&=g(x)\frac{\tau^2}{n}\Tr(\Sigma^2)+2\sqrt{2}\sqrt{\frac{\tau^2}{n}}\rho(\Sigma)\sqrt{x}\|\Sigma^{\frac{1}{2}}\sigma\|_2,\\
\label{eqn:T3}	\mathbf{T}_3(x)& =g(x)\frac{\tau^2}{n}\Tr(\Sigma^2),
\end{align}
with 
\begin{equation}
\label{eqn:C}
g(x)=1+2x+2\sqrt{x}.
\end{equation}

The following proposition will be the core of the proof of our main results. It provides deviation results on the main quantities of interest. 

\begin{proposition}\label[prop]{Prop:Majoration3Termes}
For any $0<\delta<1$, let $(\mathcal{A}_{i,\delta})_{i=1}^3$ the events respectively defined as
\begin{align}
\label{eq11} \mathcal{A}_{1,\delta} &= \left\lbrace	\left| \hat{\sigma}^T\hat{\sigma} - \sigma^T\sigma \right| \leq \mathbf{T}_1(x_{\delta})\right\rbrace ,\\
\label{eq12} \mathcal{A}_{2,\delta} &= \left\lbrace	 \left|  \hat{\sigma}^T\Sigma\hat{\sigma} - \sigma^T\Sigma\sigma \right| \le \mathbf{T}_2(x_{\delta})\right\rbrace,\\
\label{eq13} \mathrm{and} \quad \mathcal{A}_{3,\delta} &= \left\lbrace		(\hat{\sigma}-\sigma)^T\Sigma(\hat{\sigma}-\sigma)\le \mathbf{T}_3(x_{\delta}) \right\rbrace,
\end{align}
with $x_{\delta}=\ln(5/\delta)$. Then, 
$$ \mathbb{P}(\mathcal{A}_\delta) \geq 1-\delta \quad \mathrm{where} \quad \mathcal{A}_\delta:= \mathcal{A}_{1,\delta}\cap \mathcal{A}_{2,\delta} \cap \mathcal{A}_{3,\delta}.$$ 
\end{proposition}
\begin{rem}
The bounds displayed in \cref{Lemme:ExtensionLaurent} are sharp in some sense (up to the constant terms)  and  are particularly useful for testing issues. Since we are working in an estimation context, we do not need such a level of accuracy. The quantities $\mathbf{T}_1(x_{\delta})$, $\mathbf{T}_2(x_{\delta})$ and $\mathbf{T}_3(x_{\delta})$ respectively defined in \eqref{eq11}, \eqref{eq12}, \eqref{eq13} hence correspond to rough bounds on the deviation terms. In particular, in the proof below, we often the use of the inequality $\Tr(\Sigma^2) \leq \Tr^2(\Sigma)$ which is not sharp but well suited for our analysis.
\end{rem}

\begin{proof}
First, applying item i) of \cref{Lemme:ExtensionLaurent} on the variable $\hat\sigma$ with $s=0$, $t=\frac{\tau^2}{n}$, $m=\sigma$ and $A=\Sigma$, we get
$$  \PP\bigg(\hat{\sigma}^T\hat{\sigma}\ge \sigma^T\sigma+\frac{\tau^2}{n}\Tr(\Sigma)+2\sqrt{x_{\delta}}\sqrt{\frac{\tau^4}{n^2}\Tr(\Sigma^2) +2 \frac{\tau^2}{n}\rho(\Sigma)\lVert\sigma\rVert_2^2}+2{\frac{\tau^2}{n}}\rho(\Sigma)x_{\delta}\bigg)\le \frac{\delta}{5}.$$

Then, remark that 
$$ \frac{\tau^2}{n}\Tr(\Sigma)+2\sqrt{x_{\delta}}\sqrt{\frac{\tau^4}{n^2}\Tr(\Sigma^2) +2 \frac{\tau^2}{n}\rho(\Sigma)\lVert\sigma\rVert_2^2}+2\frac{\tau^2}{n}\rho(\Sigma)x_{\delta} \leq \mathbf{T}_1(x_\delta),$$
where we have used the bounds $\sqrt{\Tr(\Sigma^2)}\leq \Tr(\Sigma)$, $\rho(\Sigma) \leq \Tr(\Sigma)$ and $\sqrt{a+b} \leq \sqrt{a}+\sqrt{b}$ for any $a,b\in \mathbb{R}_+$. Similarly, 
$$ \PP\bigg(\hat{\sigma}^T\hat{\sigma}\leq \sigma^T\sigma+\frac{\tau^2}{n}\Tr(\Sigma)-2\sqrt{x_{\delta}}\sqrt{\frac{\tau^4}{n^2}\Tr(\Sigma^2) +2 \frac{\tau^2}{n}\rho(\Sigma)\lVert\sigma\rVert_2^2}\bigg)\le \frac{\delta}{5}.$$
Using again a rough bound 
$$ -\frac{\tau^2}{n}\Tr(\Sigma)+2\sqrt{x_{\delta}}\sqrt{\frac{\tau^4}{n^2}\Tr(\Sigma^2) +2 \frac{\tau^2}{n}\rho(\Sigma)\lVert\sigma\rVert_2^2} \leq \mathbf{T}_1(x_\delta),$$ 
we get
$$\mathbb{P}(\mathcal{A}^c_{1,\delta})\leq \frac{2\delta}{5}.$$
Using again \cref{Lemme:ExtensionLaurent} with $s=1$, $t=\frac{\tau^2}{n}$, $A=\Sigma$ and $m=\sigma$ (resp. $m=0$), we obtain respectively 
$$ \mathbb{P}(\mathcal{A}^c_{2,\delta})\leq \frac{2\delta}{5} \quad \mathrm{and} \quad \mathbb{P}(\mathcal{A}^c_{3,\delta})\leq \frac{\delta}{5}.$$
Using the union bound 
$$ \mathbb{P}(\mathcal{A}^c) \leq \mathbb{P}(\mathcal{A}^c_{1,\delta}) + \mathbb{P}(\mathcal{A}^c_{2,\delta}) + \mathbb{P}(\mathcal{A}^c_{3,\delta}),$$
allows to conclude the proof. 
\end{proof}

\subsection[Control of the inverse of the intensity]{Control of $\hat\lambda^{-1}$}
Recall that the single component PLS estimator of $\beta$ can be written as
$$ \hat\beta = \hat\lambda^{-1} \hat\sigma \quad \mathrm{with} \quad \hat\lambda^{-1} = \frac{\hat\sigma^T\hat\sigma}{\hat\sigma^T \Sigma \hat\sigma}.$$
To get a bound for the prediction loss, we need to control the deviation of $\hat\lambda^{-1}$ with respect to its deterministic counterpart
$$\lambda^{-1} = \frac{\sigma^T\sigma}{\sigma^T \Sigma \sigma}.$$
The stochastic term $\hat\sigma^T \Sigma \hat\sigma$ in the denominator creates statistical issues. To get rid of them, we need a minimal value for $\sigma^T \Sigma \sigma$ as presented in the following proposition. 

\begin{lemma}\label[lem]{Lem:Lambda-Lambda-1}
Assume that 
\begin{equation} 
r\, \sigma^T\Sigma\sigma \geq   g(x_\delta) \frac{\tau^2}{n} \rho(\Sigma) \Tr(\Sigma),
\label{eq:signal}
\end{equation}
for some $r\in (0,1)$. Then, on the event $\mathcal{A}_\delta$ defined in \cref{Prop:Majoration3Termes}, we have 
	$$\hat{\lambda}^{-1}\lambda\le  C_{\delta,r}, $$
for some explicit constant $C_{\delta,r}$ depending only on $\delta$ and $r$. 
\end{lemma}

\begin{proof}
First write 
$$\hat{\lambda}^{-1}\lambda=\underbrace{\frac{\hat{\sigma}^T\hat{\sigma}}{\sigma^T\sigma}}_{:=S_1} \times \underbrace{\frac{\sigma^T\Sigma\sigma}{\hat{\sigma}\Sigma\hat{\sigma}}}_{:=S_2}.$$ 
We first concentrate our attention on $S_2$. For any $s \in (0,1),$
	\begin{align*}
	    -2\sigma^T\Sigma(\sigma-\hat{\sigma})&\ge -2\sqrt{\sigma^T\Sigma\sigma}\sqrt{(\hat{\sigma}-\sigma)^T\Sigma(\hat{\sigma}-\sigma)},\\
	    &\ge -s\sigma^T\Sigma\sigma-s^{-1}(\hat{\sigma}-\sigma)^T\Sigma(\hat{\sigma}-\sigma),
	\end{align*}
since $2ab\le sa^2+s^{-1}b^2$ for all $a,b>0$. Then, on the event $\mathcal{A}_\delta$ introduced in \cref{Prop:Majoration3Termes},
	\begin{align*}
		\hat{\sigma}^T\Sigma\hat{\sigma}
		&=(\hat{\sigma}-\sigma)^T\Sigma(\hat{\sigma}-\sigma)-2\sigma^T\Sigma(\sigma-\hat{\sigma})+\sigma^T\Sigma\sigma,\\
		&\ge (\hat{\sigma}-\sigma)^T\Sigma(\hat{\sigma}-\sigma)-s\sigma^T\Sigma\sigma-s^{-1}(\hat{\sigma}-\sigma)^T\Sigma(\hat{\sigma}-\sigma)+\sigma^T\Sigma\sigma,\\
		&\ge (1-s)\sigma^T\Sigma\sigma+(1-s^{-1})(\hat{\sigma}-\sigma)^T\Sigma(\hat{\sigma}-\sigma),\\
		&\ge (1-s)\sigma^T\Sigma\sigma-(s^{-1}-1) \mathbf{T}_3(x_\delta),
	\end{align*}
where for the last two lines, we have used the fact that $s\in(0,1)$. Provided \eqref{eq:signal} holds, we obtain, since $\Tr(\Sigma^2) \leq \rho(\Sigma) \Tr(\Sigma)$,
$$ \hat{\sigma}^T\Sigma\hat{\sigma} \geq (1-s)\sigma^T\Sigma\sigma-(s^{-1}-1)r \sigma^T\Sigma\sigma = 
(1+r-s-s^{-1}r)\sigma^T\Sigma\sigma.$$
In order to obtain a positive bound, parameters $s$ and $r$ have to satisfy $(1-s)>(s^{-1}-1)r$, which holds as soon as $r<s$. Setting for instance $s=\frac{1+r}{2}$, we get 
	\begin{equation}\label{eqn:sSs}
	   \frac{\sigma^T\Sigma\sigma}{\hat{\sigma}\Sigma\hat{\sigma}} \leq  \frac{1}{1+r-(\frac{1+r}{2}+\frac{2r}{1+r})}.
	\end{equation}
Now, we provide a bound on the term $S_1$. 	Still on the event $\mathcal{A}_\delta$, we have
$$   \frac{\hat{\sigma}^T\hat{\sigma}}{\sigma^T\sigma}\le\frac{\sigma^T\sigma+\mathbf{T}_1(x_\delta)}{\sigma^T\sigma} = 1 + \frac{\mathbf{T}_1(x_\delta)}{\sigma^T\sigma}.$$
Using the simple bound $\sigma^T\Sigma \sigma \leq \sigma^T \sigma \rho(\Sigma)$, we get that the constraint \eqref{eq:signal} induces the inequality
$$r\sigma^T \sigma \geq g(x_\delta) \frac{\tau^2}{n} \Tr(\Sigma).$$
Using the definition of $\mathbf{T}_1(x_\delta)$, we hence obtain 
\begin{equation} 
\frac{\hat{\sigma}^T\hat{\sigma}}{\sigma^T\sigma} \leq 1+r+2\sqrt{2r\,\frac{x_\delta}{g(x_\delta)}}. 
\label{eqn:ss}
\end{equation}
The conclusion follows from \eqref{eqn:sSs} and \eqref{eqn:ss} with 
\begin{equation}
C_{\delta,r} =\frac{1+r+2\sqrt{2r\,\frac{x_\delta}{g(x_\delta)}}}{1+r-(\frac{1+r}{2}+\frac{2r}{1+r})}.
\label{eq:cdeltar}
\end{equation}
\end{proof}

\section[Proof of Theorem 3.1]{Proof of \cref{Th:M1}}
\label{s:proof1}

Introduce the parameter $\bar \beta$ as
 \begin{equation}\label{eqn:betabar}
 \bar{\beta}=\displaystyle\frac{\sigma^T {\sigma}}{\sigma^T \Sigma\sigma}{\sigma}:={\lambda}^{-1}\sigma = \argmin_{u\in [\sigma]} \| X\beta - X u \|^2.
 \end{equation}
Using a simple bound, we can first notice that 
\begin{equation}
\frac{1}{n} \| X(\hat\beta_\delta - \beta) \|^2 \leq \frac{2}{n} \| X(\bar\beta - \beta) \|^2 + \frac{2}{n}  \| X(\hat\beta_\delta - \bar\beta) \|^2.
\label{eq:1stbound}
\end{equation}
The first term in the r.h.s. of \eqref{eq:1stbound} exactly corresponds to the bias term appearing in \cref{Th:M1}. Let
$$ \hat\beta_{\delta,r} = \frac{\hat\sigma^T {\hat\sigma}}{\hat\sigma^T \Sigma\hat\sigma}{\hat\sigma} \times \Psi_{\delta,r} \quad \mathrm{with} \quad \Psi_{\delta,r} = \mathds{1}_{\lbrace \hat\sigma^T \Sigma\hat\sigma > t_{\delta,r} p_n \rbrace},$$
where $t_{\delta,r}$ is a threshold defined below in \eqref{eq:cdelta} and $r$ is some parameter. We choose to keep this parameter here, to highlight the test-like procedure. In particular, $r$ determines the low signal set (see below). Actually, \cref{Th:M1} will be provided by taking $r=1/2$. That is, with $ \hat\beta_{\delta}=\hat\beta_{\delta,1/2}$ and $\mathrm{t}_\delta = t_{\delta,1/2}$.

Until the end of the proof, we will focus our attention on the second term in the r.h.s. of \eqref{eq:1stbound}. To this end, we consider three different scenarios: 
\begin{itemize}
    \item[i)] The low signal case: the norm $\|\sigma\|_2$ is small (in a sense which is made precise latter on). In such a case, our estimator $\hat\beta_{\delta,r}$ is equal to $0$ with high probability (w.h.p.). 
    \item[ii)] The high signal case: the norm of $\sigma$ exceeds a given level. Then, the indicator function in $\hat\beta_{\delta,r}$ is equal to 1 w.h.p. and we use deviation results established in the previous section to obtain a bound in prediction. 
    \item[iii)] The intermediate case: we are not able to control w.h.p. the behaviour of the indicator function, but we take advantage of the bounds obtained in both previous cases. 
\end{itemize}

\subsection{Low signal case (Case i)}

We introduce $\S_L(C)$ where, for any $C \in \mathbb{R}^+$,
$$ \S_L(C)=\{\sigma\in \RR^p, \sigma^T\Sigma\sigma\le C p_n\} \quad \mathrm{and} \quad p_n = \frac{\tau^2}{n} \rho(\Sigma) \Tr(\Sigma).$$
We assume in this section that $\sigma \in \S_L(g(x_\delta)/r)$ where $r\in (0,1)$ appears in \cref{Lem:Lambda-Lambda-1}. Firstly, we prove that in this case $\Psi_{\delta,r} = 0$ on the event $\mathcal{A}_\delta$. Indeed, according to \eqref{eqn:T2} and \eqref{eq12}, we have
\begin{align*}
\hat\sigma^T \Sigma \hat\sigma
& \leq  2\sigma^T \Sigma \sigma + g(x_\delta) \frac{\tau^2}{n} \Tr(\Sigma^2) + 2 \frac{\tau^2}{n} \rho(\Sigma)^2 x_\delta \\
& \leq  2 \frac{g(x_\delta)}{r} \frac{\tau^2}{n} \rho(\Sigma) \Tr(\Sigma) + g(x_\delta) \frac{\tau^2}{n} \Tr(\Sigma^2) + 2\frac{\tau^2}{n} \rho(\Sigma)^2 x_\delta \\
& \leq \left[  2 \frac{g(x_\delta)}{r} + g(x_\delta) + 2x_\delta \right] \frac{\tau^2}{n} \rho(\Sigma) \Tr(\Sigma) \\
& = t_{\delta,r} \frac{\tau^2}{n} \rho(\Sigma) \Tr(\Sigma),
\end{align*}
where $t_{\delta,r}$ is defined as 
\begin{equation}
    t_{\delta,r} = 2 \frac{g(x_\delta)}{r} + g(x_\delta) + 2x_\delta.
    \label{eq:cdelta}
\end{equation}
This entails that $\Psi_{\delta,r} = 0$ and $\hat\beta_{\delta,r} = 0$. Using \eqref{eq:1stbound} and taking advantage of  $\sigma \in \S_L(g(x_\delta)/r)$, we hence get 
\begin{align}
\frac{1}{n} \| X\hat\beta_\delta - X\beta\|^2 
& \leq \frac{2}{n} \| X(\bar\beta - \beta) \|^2 + \frac{2}{n} \| X\bar\beta\|^2 \nonumber \\
& = \frac{2}{n} \| X(\bar\beta - \beta) \|^2 + 2\left( \frac{\sigma^T\sigma}{\sigma^T\Sigma\sigma}\right)^2 \sigma^T \Sigma \sigma \nonumber \\
& \leq \frac{2}{n} \| X(\bar\beta - \beta) \|^2 + 2 \frac{g(x_\delta)}{r} \frac{\tau^2}{n} \frac{\rho(\Sigma)\Tr(\Sigma)}{\lambda^2}.
\label{eq:boundbar}
\end{align}
This concludes the proof in this first specific regime. 

\subsection{High signal case (Case ii)}
For any $C\in \mathbb{R}^+$, let us define
$$ \S_H(C) = \{\sigma\in \RR^p, \sigma^T\Sigma\sigma \geq C p_n\} \quad \mathrm{where} \quad p_n = \frac{\tau^2}{n} \rho(\Sigma) \Tr(\Sigma). $$
We assume in this section that $\sigma \in \S_H(h_{\delta,r})$ where 
$$h_{\delta,r} = 2 \left( t_{\delta,r} + g(x_\delta) + 4x_\delta \right).$$
In this specific regime, $\Psi_{\delta,r} = 1$ on the event $\mathcal{A}_\delta$. Indeed, using first \eqref{eqn:T2} and \eqref{eq12}, we obtain 
\[
\hat\sigma^T \Sigma \hat\sigma \geq \sigma^T \Sigma \sigma - g(x_\delta) \frac{\tau^2}{n} \Tr(\Sigma^2) - 2\,\bigl(4 x_\delta \frac{\tau^2}{n} \rho(\Sigma)^2\bigr)^{1/2}\bigl(\sigma^T\Sigma\sigma/2\bigr)^{1/2} \]
Since for all $a,b\in\mathbb{R}$, $a^2-2ab\geq -b^2$, we get
\begin{align*}
\hat\sigma^T \Sigma \hat\sigma 
& \geq \frac{1}{2} \sigma^T \Sigma \sigma - g(x_\delta) \frac{\tau^2}{n} \Tr(\Sigma^2) - 4 x_\delta \frac{\tau^2}{n} \rho(\Sigma)^2 \\
& \geq \frac{h_{\delta,r}}{2} \frac{\tau^2}{n} \rho(\Sigma) \Tr(\Sigma) - g(x_\delta) \frac{\tau^2}{n} \Tr(\Sigma^2) - 4 x_\delta \frac{\tau^2}{n} \rho(\Sigma)^2\\
& \geq \left( \frac{h_{\delta,r}}{2} - g(x_\delta) - 4 x_\delta\right) \frac{\tau^2}{n} \rho(\Sigma) \Tr(\Sigma)\\
& = t_{\delta,r} \frac{\tau^2}{n} \rho(\Sigma) \Tr(\Sigma),
\end{align*}
according to the definition of $h_{\delta,r}$. This entails in particular that, on the event $\mathcal{A}_\delta$, $\Psi_{\delta,r}=1$ and
$$ \hat\beta_{\delta,r} = \frac{\hat\sigma^T {\hat\sigma}}{\hat\sigma^T \Sigma\hat\sigma}{\hat\sigma} = \hat\lambda^{-1}\hat\sigma= \hat\beta_{PLS}. $$
According to \eqref{eq:1stbound}, we have
$$\frac{1}{n} \| X(\hat\beta_{\delta,r} - \beta) \|^2 \leq \frac{2}{n} \| X(\bar\beta - \beta) \|^2 + \frac{2}{n}  \| X(\hat\beta - \bar\beta) \|^2.$$
In the following, we study the second term in the r.h.s. of the previous equality. First remark that
$$\hat{\beta}_{PLS}-\bar{\beta}=\hat{\lambda}^{-1}\hat{\sigma}-\lambda^{-1}\sigma=\hat{\lambda}^{-1}(\hat{\sigma}-\sigma)-\hat{\lambda}^{-1}\frac{\hat{\sigma}^T\Sigma\hat{\sigma}-\sigma^T\Sigma\sigma}{\sigma^T\Sigma\sigma}\sigma+\frac{\hat{\sigma}^T\hat{\sigma}-\sigma^T\sigma}{\sigma^T\Sigma\sigma}\sigma.$$

It yields
\begin{equation}\frac{2}{n}\|X(\bar{\beta}-\hat{\beta}_{PLS})\|^2
\le \frac{4}{n}\hat{\lambda}^{-2}\|X(\hat{\sigma}-\sigma)\|^2+{8}\,\hat{\lambda}^{-2}\frac{(\hat{\sigma}^T\Sigma\hat{\sigma}-\sigma^T\Sigma\sigma)^2}{\sigma^T\Sigma\sigma}+{8}\,\frac{(\hat{\sigma}^T\hat{\sigma}-\sigma^T\sigma)^2}{\sigma^T\Sigma\sigma}.
\label{eq:decompo}
\end{equation}

Remark that $\hat{\lambda}^{-2}=(\hat{\lambda}^{-1}\lambda)^{2}\lambda^{-2}.$ 
The term $(\hat{\lambda}^{-1}\lambda)$ illustrates the ratio between the estimation of the norm of the first PLS component and its theoretical value. 
Using \cref{Lem:Lambda-Lambda-1}, still on the event $\mathcal{A}_\delta$, we have 
$$\hat{\lambda}^{-1}\lambda\leq C_{\delta,r},$$
where $C_{\delta,r}$ defined in \eqref{eq:cdeltar}. Consequently, 
$$
	\frac{2}{n}\|X(\bar{\beta}-\hat{\beta}_{PLS})\|^2 \le 4C_{\delta,r}^2 \frac{1}{n}\lambda^{-2}\|X(\hat{\sigma}-\sigma)\|^2+8 C_{\delta,r}^2 \lambda^{-2}\frac{(\hat{\sigma}^T\Sigma\hat{\sigma}-\sigma^T\Sigma\sigma)^2}{\sigma^T\Sigma\sigma}+8\frac{(\hat{\sigma}^T\hat{\sigma}-\sigma^T\sigma)^2}{\sigma^T\Sigma\sigma}.
$$

A direct application of \cref{Lem:Lambda-Lambda-1} entails that
\begin{equation}
	\frac{2}{n}\|X(\bar{\beta}-\hat{\beta})\|^2 \le 4 C_{\delta,r}^2 \lambda^{-2}\mathbf{T}_3(x_\delta)+8C_{\delta,r}^2 \lambda^{-2}\frac{\mathbf{T}_2(x_\delta)^2}{\sigma^T\Sigma\sigma}+8\frac{\mathbf{T}_1(x_\delta)^2}{\sigma^T\Sigma\sigma}.
	\label{eq:inter1}
\end{equation}

Using \eqref{eqn:T2} and the fact that $\sigma \in \S_H(h_{\delta,r})$, we first get
\begin{align}
\frac{\mathbf{T}_2(x_\delta)^2}{\sigma^T\Sigma\sigma} 
& \leq 2 g(x_\delta)^2 \left(\frac{\tau^2}{n}\Tr(\Sigma^2)\right)^2 \times \frac{1}{\sigma^T\Sigma \sigma} + 16 x_\delta \frac{\tau^2}{n} \rho^2(\Sigma), \nonumber \\
& \leq 2\frac{g(x_\delta)^2}{h_{\delta,r}} \frac{\tau^2}{n} \rho(\Sigma)\Tr(\Sigma) + 16 x_\delta \frac{\tau^2}{n} \rho(\Sigma)\Tr(\Sigma), \nonumber \\
& =  \left( 2\frac{g(x_\delta)^2}{h_{\delta,r}} + 16 x_\delta\right) \frac{\tau^2}{n} \rho(\Sigma)\Tr(\Sigma).
\label{eq:inter2}
\end{align}
Similarly, using \eqref{eqn:T1} and $\sigma \in \S_H(h_{\delta,r})$, we obtain
\begin{align}
\frac{\mathbf{T}_1(x_\delta)^2}{\sigma^T\Sigma\sigma}    
& \leq 2 g(x_\delta)^2 \left(\frac{\tau^2}{n}\Tr(\Sigma)\right)^2 \times \frac{1}{\sigma^T\Sigma \sigma} + 16 x_\delta \frac{\tau^2}{n} \rho(\Sigma)\times \lambda^{-1}, \nonumber \\
& \leq 2\frac{g(x_\delta)^2}{h_{\delta,r}} \frac{\tau^2}{n} \Tr(\Sigma) \times \frac{1}{\rho(\Sigma)} + 16 x_\delta \frac{\tau^2}{n} \rho(\Sigma)\times \lambda^{-1}, \nonumber \\
& \leq \left( 2\frac{g(x_\delta)^2}{h_{\delta,r}} + 16 x_\delta \right) \frac{\tau^2}{n} \Tr(\Sigma) \times \lambda^{-1},
\label{eq:inter3}
\end{align}
since $\lambda \leq \rho(\Sigma)$. Gathering \eqref{eq:inter1}, \eqref{eq:inter2} and \eqref{eq:inter3}, we obtain
\begin{multline*}
{\frac{2}{n}\|X(\bar{\beta}-\hat{\beta})\|^2 }
 \leq \left( 4 C_{\delta,r}^2 + 16 C_{\delta,r} \frac{g(x_\delta)^2}{h_{\delta,r}} + 128 C_{\delta,r} x_\delta\right) \lambda^{-2}  \frac{\tau^2}{n} \rho(\Sigma) \Tr(\Sigma) \\
 +  \left( 16 \frac{g(x_\delta)^2}{h_{\delta,r}} + 128 x_\delta \right) \frac{\tau^2}{n} \Tr(\Sigma) \times \lambda^{-1} ,
\end{multline*}
which provides a bound in this regime.

\subsection{Intermediate case (Case iii)}
We finish the proof with the scenario where $\sigma \in \S_I$ with 
$$ \S_I = \left\{\sigma\in \RR^p, \frac{g(x_\delta)}{r} p_n \leq \sigma^T\Sigma\sigma \leq h_{\delta,r} p_n\right\}.$$ Remark that $\S_I$ depends on $\delta$ and $r$, but we omit the dependence in the notation.
We stress that $\S_I = \S_L(g(x_\delta)/r)^c \cap \S_H(h_\delta)^c$. This section hence covers all the situations that have not been considered before. The fact that $\sigma^T \Sigma \sigma$ is both bounded from above and below allows to control the risk in prediction on $\mathcal{A}_\delta$ whatever the value of the indicator function $\Psi_{\delta,r}$ is. Indeed, starting form \eqref{eq:1stbound}, we have
\begin{align*}
\frac{1}{n} \| X(\hat\beta -\beta)\|^2 
& \leq \frac{2}{n} \| X(\bar\beta -\beta)\|^2 + \frac{2}{n} \| X(\hat\beta_{\delta,r} -\bar\beta)\|^2 \mathds{1}_{\lbrace \Psi_{\delta,r} =0 \rbrace} + \frac{2}{n} \| X(\hat\beta_{\delta,r} -\bar\beta)\|^2 \mathds{1}_{\lbrace \Psi_{\delta,r} \neq 0 \rbrace}, \\
& \leq \frac{2}{n} \| X(\bar\beta -\beta)\|^2 + \frac{2}{n} \| X\bar\beta\|^2  + \frac{2}{n} \| X(\hat\beta_{PLS} -\bar\beta)\|^2.
\end{align*}
Using $\S_I \subset \S_L(h_{\delta,r})$ and the same algebraic method as in \eqref{eq:boundbar}, we first obtain 
$$ \| X\bar\beta\|^2 \leq h_{\delta,r} \frac{\tau^2}{n} \frac{\rho(\Sigma)\Tr(\Sigma)}{\lambda^2}. $$
Similarly, using $\S_I \subset \S_H(g(x_\delta)/r)$ and the same bounds as those displayed between \eqref{eq:decompo} and \eqref{eq:inter3}, we get
\begin{multline*}
{\frac{2}{n}\|X(\bar{\beta}-\hat{\beta})\|^2 }
 \leq   \left( 4 C_{\delta,r}^2 + 16 C_{\delta,r} r g(x_\delta) + 128 C_{\delta,r} x_\delta\right) \lambda^{-2}  \frac{\tau^2}{n} \frac{\rho(\Sigma) \Tr(\Sigma)}{\lambda^2} \\  +  \left( 16r g(x_\delta) + 128 x_\delta \right) \frac{\tau^2}{n} \frac{\Tr(\Sigma)}{\lambda}.
\end{multline*}

\subsection{End of the proof}

Considering the three cases i), ii) and iii), we get 
\begin{align*}
\lefteqn{\frac{2}{n}\|X({\beta}-\hat{\beta}_{\delta,r})\|^2 }\\ 
&\leq \frac{2}{n} \| X(\bar\beta -\beta)\|^2  \\
& \quad +  \left( 2\frac{g(x_\delta)}{r}+4\, C_{\delta,r}^2 + 16\, C_{\delta,r}\, r\, g(x_\delta)\max(1,\frac{g(x_\delta)}{r\,h_{\delta,r}}) + 128\, C_{\delta,r}\, x_\delta\right) \frac{\tau^2}{n} \frac{\rho(\Sigma) \Tr(\Sigma)}{\lambda^2} \\
 & \quad +  \left( 16\,r\, g(x_\delta)\max(1,\frac{g(x_\delta)}{r\,h_{\delta,r}}) + 128\, x_\delta \right) \frac{\tau^2}{n} \frac{\Tr(\Sigma)}{\lambda}, \\
 &:= \frac{2}{n} \| X(\bar\beta -\beta)\|^2 + C_{\delta,r}' \frac{\tau^2}{n} \max\left(\frac{\Tr(\Sigma)}{\lambda}, \frac{\rho(\Sigma)\Tr(\Sigma)}{\lambda^2} \right).
\end{align*}
As stated previously, the hyperparameter $r$ can be fixed to any value in $(0,1)$. The calibration of $r$ may be used to improve the constants obtained in the above bound. As we are not looking for optimal constants, we can choose an arbitrary value of $r$. \cref{Th:M1} follows considering e.g. $r=1/2$, that is, $\hat\beta_\delta=\hat\beta_{\delta,1/2}$, $C_\delta=C_{\delta,1/2}'$.

\section{Technical results - The sparse case}
\label{s:asparse}
This section contains technical results that will be useful for the proofs of the results presented in \cref{s:sparse}. To this end, we need some additional notation. For any $x\in\mathbb{R}$, we write
$$ \mathrm{sgn}(x) = \left\lbrace
\begin{array}{ccc}
1 & \mathrm{if} & x>0, \\
-1 & \mathrm{if} & x<0, \\
0 & \mathrm{if} & x=0.
\end{array} \right.
$$
and $x_+ = x\mathbf{1}_{\lbrace x\geq 0\rbrace}$. 
In the next two parts, we will write $|\J_0|=\Tr(\Sigma_{\J_0})$ according to \ref{ass:A1}. 

\subsection{Deviation inequalities}
We start with a proposition that lays out the estimator $\hat\beta_{sPLS}$ explicitly. \citet{Durif} provide a proof with a closed form, for the sake of clarity we reproduce a proof below.

\begin{proposition}\label{Prop:sPLS}
The solution $\tilde w_1$ of the optimisation problem \eqref{Eq:Opti:sPLS} satisfies
$$ \tilde w_1 = \frac{\tilde\sigma}{\| \tilde\sigma\|_2} \quad \mathrm{with} \quad \tilde\sigma_j = \mathrm{sgn}(\hat\sigma_j)(|\hat\sigma_j| - \mu)_+ \quad \forall j\in\lbrace 1,\dots, p\rbrace.$$
Moreover, $\hat\beta_{sPLS}$ defined in \eqref{eq:betaPLS} with $W=\tilde w_1$ can be written as
$$\hat\beta_{sPLS} = \frac{\tilde\sigma^T\hat\sigma}{\tilde\sigma^T\Sigma\tilde\sigma} \times \tilde\sigma.$$
\end{proposition}
\begin{proof}
Applying the method of Lagrange multipliers, the optimisation problem \eqref{Eq:Opti:sPLS} becomes
\begin{equation}\label{Eq:Opti:sPLS2}
\argmin_{w\in\RR^p, \nu>0}\big\{-\hat{\sigma}^Tw+\nu(\|w\|_2^2-1)+\mu\|w\|_1\big\}
\end{equation}
This problem writes as $\min\{f(w)+g(w)\}$ with $f: w \mapsto -\hat{\sigma}^Tw+\nu(\|w\|_2^2-1)$ a differentiable convex function and $g: w\mapsto \mu\|w\|_1$ a convex function. Then, $w^*$ is a global optimal solution of \eqref{Eq:Opti:sPLS2} if and only if
$$-\nabla f(w^*)\in\partial g(w^*),$$ 
where $\partial g$ is the sub-differential of the function $g$ (see, e.g., \citep{Bach} for more details).  This condition can be written as
\begin{equation}
-\frac{\nabla f}{\mu}(w^*)\in\partial {\|\cdot\|_1}(w^*) \quad \Leftrightarrow \quad \frac{\hat{\sigma}-\rho w^*}{\mu}\in \partial {\|\cdot\|_1}(w^*)
\label{eq:opti_cond}
\end{equation}
where in the formula above, $\|.\|_1$ denotes the function $w\mapsto \|w\|_1$. If $w^*_i=0$, condition \eqref{eq:opti_cond} leads to $\hat{\sigma}_i\in[-\mu,\mu]$. If $w^*_i>0$, we obtain $\hat{\sigma}_i-\nu w^*_i=\mu$, which implies $w^*_i=\frac{\hat{\sigma}_i-\mu}{\nu}$. In a similar way for $w_i^*<0$ we have $w_i^*=\frac{\hat{\sigma}+\mu}{\nu}$. We deduce that if $w^*_i\ne0$, then $\mathrm{sgn}(w^*_i)=\mathrm{sgn}(\hat{\sigma}_i).$ 

These computations lead to the closed form  $$w^*_i=\frac{(|\hat{\sigma}_i|-\mu)_+\mathrm{sgn}(\hat{\sigma}_i)}{\nu} \quad \forall i\in \lbrace 1,\dots, p \rbrace.$$ 
The solution $\tilde{w_1}=(w_1^*,\dots w_p^*)$ is then chosen with $\nu$ such that the solution has a unitary norm. 
We derive the formula for $\hat\beta_{sPLS}$ thanks to \cref{Algo:PLS} with weight $\tilde w_1.$ 
\end{proof}

The following proposition makes precise the relationship between the $\hat\sigma_j$ and the threshold $\mu$.

\begin{proposition}
    For any $0 < \delta <1$, let $\M_{\delta}$ the event defined as 
    $$
        \M_{\delta} = \bigcap_{j=1}^p \left\{ |\hat{\sigma}_j-\sigma_j|\leq\frac{\mu}{2}\right\},
    $$
with $\mu$ as defined in \eqref{eq:mu}. Then, under \ref{ass:A1},
$$
    \PP\big(\M_{\delta}\big)\ge 1-\frac{\delta}{2}.
$$
\end{proposition}
\begin{proof}
Using \cref{lem:variances}, $\hat\sigma_j\sim\mathcal N(\sigma_j,\frac{\tau^2}{n})$ for all $j\in\lbrace 1,\dots, p \rbrace$ since $\Sigma$ is normalized according to \ref{ass:A1}.  We get
$$\PP(|\hat{\sigma}_j-\sigma_j|>\mu)\le \frac{\delta}{2p} \quad \quad  \forall j\in\lbrace 1,\dots, p \rbrace.$$
Then,
$$\PP(\mathcal{M}_\delta^c) = \PP\left(\bigcup\limits_{j=1}^{p}\left\{|\hat{\sigma}_j-\sigma_j|>\frac{\mu}{2}\right\}\right)\le \sum_{j=1}^p\PP\left(|\hat{\sigma}_j-\sigma_j|>\frac{\mu}{2}\right)\le \frac{\delta}{2}.$$
\end{proof}

In the following, we denote by $\hat{\mathcal{J}}$ the support of $\tilde\sigma$, namely
$$ \hat{\mathcal{J}}=\{j\in\lbrace 1,\dots, p \rbrace \ \mathrm{s.t.} \ |\hat{\sigma}_j|>\mu\}.$$ The exponent $C$ will denote the complementary set \textit{i.e.} for any set $\mathcal I$, $\mathcal{I}^C=\lbrace j\in\lbrace 1,\dots, p\rbrace, j\notin\mathcal{I} \rbrace$.
On the event $\mathcal{M}_\delta$ introduced above, we are able to localize $ \hat{\mathcal{J}}$ as shown in the next lemma.

\begin{lemma}
\label{lem:support}
On the event $\M_{\delta}$ we have,
$$ \mathcal{J}_{0,2} \subset \hat{\mathcal{J}} \subset \mathcal{J}_{0,1} \subset \J_0,$$
where 
$$ \mathcal{J}_{0,1}=\left\{j\in\lbrace 1,\dots, p \rbrace \ \mathrm{s.t.} \ |\sigma_j|>\frac{\mu}{2}\right\} \quad \mathrm{and} \quad \mathcal{J}_{0,2} =\{j\in\lbrace 1,\dots, p \rbrace \ \mathrm{s.t.} \ |\sigma_j|>2\mu\}.$$
\end{lemma}
\begin{proof}[Proof]
On the event $\mathcal{M}_\delta$, for any $j\in\lbrace 1,\dots,p \rbrace$,
$$ |\sigma_j|>2\mu \quad \Rightarrow \quad |\hat\sigma_j| >\mu.$$
This entails that $\mathcal{J}_{0,2} \subset \hat{\mathcal{J}}$. Similarly, on the event $\M_{\delta}$, for any $j\in\lbrace 1,\dots,p \rbrace$,
$$|\hat\sigma_j| >\mu \quad  \Rightarrow \quad |\sigma_j|>\frac{\mu}{2},$$
which implies $\hat{\mathcal{J}} \subset \mathcal{J}_{0,1}$.
\end{proof}

By employing a similar line of reasoning as in \cref{Prop:Majoration3Termes}, we introduce the following quantities and events:
\begin{align}
\label{eqn:T1Sparse}
\mathbf{T}_{0,1}(x)&=g(x)\frac{\tau^2}{n}\Tr(\Sigma_{\J_0})+2\sqrt{2}\sqrt{\frac{\tau^2}{n}}\rho(\Sigma_{\J_0})^{\frac{1}{2}}\sqrt{x}\|\sigma\|_2,\\
\label{eqn:T2Sparse}	
\mathbf{T}_{0,2}(x)& =g(x)\frac{\tau^2}{n}\Tr(\Sigma_{\J_0}),\\
\label{eqn:T3Sparse}
\mathbf{T}_{0,3}(x)& =g(x)\frac{\tau^2}{n}\Tr(\Sigma_{\J_0}^2),
\end{align}
with 
\begin{equation}
g(x)=1+2x+2\sqrt{x}, \quad \forall x\in \mathbb{R}.
\end{equation}

The following proposition will be the core of the proof for the sparse case. It provides deviation results on the main quantities of interest. 

\begin{proposition}\label{Prop:Majoration3TermesSparse}
For any $0<\delta<1$, let $(\mathcal{B}_{i,\delta})_{i=1}^2$ the events respectively defined as
\begin{align}
\label{eq11Sparse} \mathcal{B}_{1,\delta} &= \left\lbrace	|\hat{\sigma}_{\J_0}^T\hat{\sigma}_{\J_0}-\sigma^T\sigma|\le\mathbf{T}_{0,1}(x_{0,\delta})\right\rbrace ,\\
\label{eq12Sparse} \mathcal{B}_{2,\delta} &= \left\lbrace		(\hat{\sigma}_{\J_0}-\sigma)^T(\hat{\sigma}_{\J_0}-\sigma)\le \mathbf{T}_{0,2}(x_{0,\delta}) \right\rbrace,\\
\label{eq13Sparse} \mathrm{and} \quad \mathcal{B}_{3,\delta} &= \left\lbrace		(\hat{\sigma}_{\J_0}-\sigma)^T\Sigma(\hat{\sigma}_{\J_0}-\sigma)\le \mathbf{T}_{0,3}(x_{0,\delta}) \right\rbrace,
\end{align}
with $x_{0,\delta}=\ln(10/\delta)$. Then, 
$$ \mathbb{P}(\mathcal{B}_\delta) \geq 1-\frac{\delta}{2} \quad \mathrm{where} \quad \mathcal{B}_\delta:= \mathcal{B}_{1,\delta}\cap \mathcal{B}_{2,\delta}\cap\mathcal{B}_{3,\delta}.$$ 
\end{proposition}
\begin{proof}
The proof is a straightforward generalization of \cref{Prop:Majoration3Termes} and is thus omitted.  
\end{proof}

Using a simple bound, we can easily prove that $\mathbb{P}(\mathcal{M}_\delta \cap\mathcal{B}_\delta) \geq 1- \delta$.

\subsection[Control of the approximation of lambda]{Control of $\tilde\lambda$}

In the following, we introduce
$$ \tilde \lambda = \frac{\tilde\sigma^T \Sigma \tilde\sigma}{\tilde\sigma^T\hat\sigma}.$$
The purpose of this section and of the following proposition is to prove that the ratio $\tilde \lambda^{-1}\lambda$ is controlled with high probability by a constant. The control is acquired under \ref{ass:A2}, that is, when the amount of signal on the component is high enough. It yields, hence, the calibration of the parameter $d_{\delta,p}$ in \ref{ass:A2}.

\begin{proposition}\label{Prop:LambdalambdaSparse}
Assume that \ref{ass:A2} is satisfied, namely that
\[ \sigma^T \Sigma\sigma > d_{\delta,p}\frac{\tau^2}{n}\rho(\Sigma_{\J_0})\Tr(\Sigma_{\J_0}),\] {with} $d_{\delta,p} = 4\,g(x_{0,\delta})+192 \ln(\frac{2p}{\delta}).$
Then, on the event $\M_{\delta}\cap \B_{\delta}$,
	$$\tilde{\lambda}^{-1}\lambda\le F,$$
with $F= 112$.
\end{proposition}

\begin{proof}
First, remark that 
$$ \tilde\lambda^{-1}\lambda = \frac{\tilde\sigma^T\hat\sigma}{\sigma^T\sigma} \times \frac{\sigma^T\Sigma\sigma}{\tilde\sigma^T\Sigma\tilde\sigma}.$$
To start the proof, we focus our attention on the first ratio in the r.h.s. of the previous equality. Introduce 
$$ S_j = \mathrm{sgn}(\hat\sigma_j) \quad \forall j\in\lbrace 1,\dots, p \rbrace.$$
We remark that, on the event $\M_{\delta}\cap \B_{\delta}$,
\[
	\lvert\tilde{\sigma}^T\hat{\sigma}\rvert
	\le\sum_{j\in\hat{\mathcal{J}}} \lvert\tilde{\sigma}_j\hat{\sigma}_j \rvert \le\sum_{j\in\hat{\mathcal{J}}} \lvert (\hat{\sigma}_j-\mu S_j)\hat{\sigma_j} \rvert
	\le\hat{\sigma}^T_{\J_0}\hat{\sigma}_{\J_0}+\mu\sum_{j\in\hat{\mathcal{J}}}|\hat{\sigma}_j|,\]
since $\hat{\mathcal{J}} \subset \mathcal{J}_{0,1} \subset \mathcal{J}_0$ according to \cref{lem:support}. Using these inclusions and the Cauchy-Schwarz inequality, we get
\[|\tilde{\sigma}^T\hat{\sigma}|
\le \hat{\sigma}_{\J_0}^T\hat{\sigma}_{\J_0}+\sqrt{\mu^2|\J_0|}\sqrt{\hat{\sigma}_{\J_0}^T\hat{\sigma}_{\J_0}}
	\le 2\hat{\sigma}_{\J_0}^T\hat{\sigma}_{\J_0}+\mu^2|\J_0|.
	\]
\cref{Prop:Majoration3TermesSparse} implies that
\[
	|\tilde{\sigma}^T\hat{\sigma}| \leq 2\sigma^T\sigma + 2\mathbf{T}_{0,1}(x_{0,\delta})+\mu^2|\J_0|.
	\]
Using the rough bound $\sigma^T \Sigma \sigma \leq \rho(\Sigma_{\mathcal{J}_0}) \sigma^T\sigma$, \ref{ass:A2}  with $d_{\delta,p}=4\,g(x_{0,\delta})+192 \ln(\frac{2p}{\delta})$ yields
$$ \frac{1}{4} \sigma^T\sigma \geq \left( \frac{\tau^2}{n} g(x_{0,\delta}) + 6\mu^2 \right) \Tr(\Sigma_{\mathcal{J}_0}).$$
We obtain
\begin{align}
\nonumber \tilde{\sigma}^T\hat{\sigma} & \leq 2\sigma^T\sigma + 2g(x_{0,\delta})\frac{\tau^2}{n}\Tr(\Sigma_{\J_0})+4\sqrt{2}\sqrt{\frac{\tau^2}{n}g(x_{0,\delta})} \sqrt{\Tr(\Sigma_{\J_0})}\|\sigma\|_2+\mu^2\Tr(\Sigma_{\mathcal{J}_0})\\
\nonumber & \leq 2\sigma^T\sigma + \frac{1}{2} \sigma^T\sigma + 2\sqrt{2} \sigma^T\sigma + \frac{1}{24} \sigma^T\sigma \\
 & \leq 7\sigma^T\sigma.
\label{eq:lambda_1ere}
\end{align}

Now, we turn our attention to the ratio $\sigma^T\Sigma\sigma / \tilde\sigma^T\Sigma\tilde\sigma$. We first bound the quantity using the set $\hat\J$ as follows:  
\begin{align*}
\tilde\sigma^T\Sigma \tilde\sigma
& = \tilde\sigma^T_{\hat{\mathcal{J}}}\Sigma \tilde\sigma_{\hat{\mathcal{J}}} \\
& = (\hat\sigma - \mu S)^T_{\hat{\mathcal{J}}}\Sigma (\hat\sigma-\mu S)_{\hat{\mathcal{J}}} \\
& = \hat\sigma^T_{\hat{\mathcal{J}}}\Sigma \hat\sigma_{\hat{\mathcal{J}}} +\mu^2 S^T_{\hat{\mathcal{J}}}\Sigma S_{\hat{\mathcal{J}}} - 2\mu \hat\sigma^T_{\hat{\mathcal{J}}}\Sigma S_{\hat{\mathcal{J}}}   \\
& \geq  \frac{1}{2} \hat\sigma^T_{\hat{\mathcal{J}}}\Sigma \hat\sigma_{\hat{\mathcal{J}}} - \mu^2 S^T_{\hat{\mathcal{J}}}\Sigma S_{\hat{\mathcal{J}}}\\
& \geq \frac{1}{2} \hat\sigma^T_{\hat{\mathcal{J}}}\Sigma \hat\sigma_{\hat{\mathcal{J}}} - \mu^2 \rho(\Sigma_{\hat{\J}}) |\hat{\J}|\\
& \geq \frac{1}{2} \hat\sigma^T_{\hat{\mathcal{J}}}\Sigma \hat\sigma_{\hat{\mathcal{J}}} - \mu^2 \rho(\Sigma_{\J_0}) |\J_0|,
\end{align*}
where the last inequality comes from the inclusion $\hat\J\subset \J_0$.
Now we want to find a bound relying only on the set $\J_0$. Note that
\begin{align*}
\hat\sigma^T_{\hat{\mathcal{J}}}\Sigma \hat\sigma_{\hat{\mathcal{J}}}
& = (\hat{\sigma}_{\J_0}-\hat{\sigma}_{\J_0\cap\hat{\J}^C})^T\Sigma  (\hat{\sigma}_{\J_0}-\hat{\sigma}_{\J_0\cap\hat{\J}^C}) \\
& =\hat{\sigma}_{\J_0}^T\Sigma \hat{\sigma}_{\J_0}-2\hat{\sigma}_{\J_0}^T\Sigma\hat{\sigma}_{\J_0\cap\hat{\J}^C}+\hat{\sigma}_{\J_0\cap\hat{\J}^C}^T\Sigma\hat{\sigma}_{\J_0\cap\hat{\J}^C} \\
& \ge \frac{\hat{\sigma}_{\J_0}^T\Sigma\hat{\sigma}_{\J_0}}{2}-\hat{\sigma}^T_{\J_0\cap\hat{\J}^C}\Sigma\hat{\sigma}_{\J_0\cap\hat{\J}^C}\\
& 	\ge \frac{\hat{\sigma}^T_{\J_0}\Sigma\hat{\sigma}_{\J_0}}{2}-\mu^2\mathds{1}_{\J_0\cap\hat{\J}^C}^T\Sigma_{\J_0}\mathds{1}_{\J_0\cap\hat{\J}^C}\\
&  \ge\frac{\hat{\sigma}^T_{\J_0}\Sigma\hat{\sigma}_{\J_0}}{2}-\mu^2\rho(\Sigma_{\J_0})|\J_0|.
\end{align*}
We deduce the following inequality:
\begin{align}
    \nonumber \tilde{\sigma}^T\Sigma\tilde{\sigma}&
    \ge\frac{\hat{\sigma}^T_{\J_0}\Sigma\hat{\sigma}_{\J_0}}{4}-\frac{\mu^2}{2}\rho(\Sigma_{\J_0})|\J_0|-\mu^2\rho(\Sigma_{\J_0})|\J_0|\\
    \nonumber &\ge\frac{\hat{\sigma}^T_{\J_0}\Sigma\hat{\sigma}_{\J_0}}{4}-\frac{3\mu^2}{2}\rho(\Sigma_{\J_0})|\J_0| \\
    \label{eq:sigmatilde1} & \ge \frac{1}{4}\big(\hat{\sigma}_{\J_{0}}^T\Sigma\hat{\sigma}_{\J_0}-6\rho(\Sigma_{\J_0})\mu^2|\J_0|\big).
\end{align}
Finally, we obtain a bound that does not depend on $\J_0$. Indeed, 
	\begin{align*}
		\lefteqn{\hat{\sigma}_{\J_0}^T\Sigma\hat{\sigma}_{\J_0}-6\rho(\Sigma_{\J_0})\mu^2|\J_0|}\\
		&=(\hat{\sigma}-\sigma)^T_{\J_0}\Sigma(\hat{\sigma}-\sigma)_{\J_0}+2\sigma^T\Sigma(\hat{\sigma}-\sigma)_{\J_0}+\sigma^T\Sigma\sigma-6\rho(\Sigma_{\J_0})\mu^2|\J_0|\\
		&\ge(\hat{\sigma}-\sigma)^T_{\J_0}\Sigma(\hat{\sigma}-\sigma)_{\J_0}-\frac{1}{2}\sigma^T\Sigma\sigma-2(\hat{\sigma}-\sigma)^T_{\J_0}\Sigma(\hat{\sigma}-\sigma)_{\J_0}+\sigma^T\Sigma\sigma-6\rho(\Sigma_{\J_0})\mu^2|\J_0|\\
		&\ge\frac{1}{2}\sigma^T\Sigma\sigma-(\hat{\sigma}-\sigma)^T_{\J_0}\Sigma(\hat{\sigma}-\sigma)_{\J_0}-6\rho(\Sigma_{\J_0})\mu^2|\J_0|.
	\end{align*}
Using the fact that we are on the event $\mathcal B_\delta$, we have
\begin{align}
\nonumber	{\hat{\sigma}_{\J_0}^T\Sigma\hat{\sigma}_{\J_0}-6\rho(\Sigma_{\J_0})\mu^2|\J_0|}
\nonumber	& \ge\frac{1}{2}\sigma^T\Sigma\sigma-\mathbf{T}_{0,2}(x_{0,\delta})-6\rho(\Sigma_{\J_0})\mu^2|\J_0|\\ & \nonumber\ge\frac{1}{2}\sigma^T\Sigma\sigma-g(x_{0,\delta})\frac{\tau^2}{n}\Tr(\Sigma_{\J_0}^2)-6\rho(\Sigma_{\J_0})\mu^2|\J_0|, \\
\nonumber	& \ge \frac{1}{2} \sigma^T\Sigma\sigma - \left( \frac{\tau^2}{n} g(x_{0,\delta})+ 6\mu^2\right)\rho(\Sigma_{\mathcal{J}_0}) \Tr(\Sigma_{\mathcal{J}_0}), \\
	\label{eq:sigmatilde2} & \ge \frac{1}{4} \sigma^T\Sigma\sigma, 
\end{align}
where we have used \ref{ass:A2} for the last inequality, with $d_{\delta,p} = 4\,g(x_{0,\delta})+192 \ln(\frac{2p}{\delta}).$

We deduce from \eqref{eq:sigmatilde1} and \eqref{eq:sigmatilde2} that
\begin{equation}
\label{eq:lambda_2nde}
\tilde\sigma^T\Sigma \tilde\sigma \ge \frac{1}{16}\sigma^T\Sigma\sigma.
\end{equation}	
	
Inequalities \eqref{eq:lambda_1ere} and \eqref{eq:lambda_2nde} lead to 	
$$\tilde{\lambda}^{-1}\lambda\le 7\times 16,$$
which concludes the proof.
\end{proof}

\section[Proof of the results displayed in Section 4]{Proof of the results displayed in \cref{s:sparse}}
\label{s:lasta}

\subsection[Proof of Theorem 4.1]{Proof of \cref{Th:M2}}
\label{s:proof2}
From now on we work on the event $\M_{\delta}\cap\B_{\delta}$. First
$$\frac{1}{n}\|X\hat\beta_{sPLS}-X\beta\|^2\le\frac{2}{n}\|X(\hat\beta_{sPLS}-\bar{\beta})\|^2+\frac{2}{n}\|X(\bar{\beta}-\beta)\|^2,$$ with $\bar{\beta}=\frac{\sigma^T\sigma}{\sigma^T\Sigma\sigma}\sigma.$
We use again the segmentation proposed in the proof of \cref{Th:M1},
$$\hat\beta_{sPLS}-\bar{\beta}=\tilde{\lambda}^{-1}\tilde{\sigma}-\lambda^{-1}\sigma=\tilde{\lambda}^{-1}(\tilde{\sigma}-\sigma)-\tilde{\lambda}^{-1}\frac{\tilde{\sigma}^T\Sigma\tilde{\sigma}-\sigma^T\Sigma\sigma}{\sigma^T\Sigma\sigma}\sigma+\frac{\tilde{\sigma}^T\hat{\sigma}-\sigma^T\sigma}{\sigma^T\Sigma\sigma}\sigma.$$

It yields
\begin{equation}\frac{2}{n}\|X(\bar{\beta}-\hat\beta_{sPLS})\|^2
\le \frac{4}{n}\tilde{\lambda}^{-2}\|X(\tilde{\sigma}-\sigma)\|^2+{8}\,\tilde{\lambda}^{-2}\frac{(\tilde{\sigma}^T\Sigma\tilde{\sigma}-\sigma^T\Sigma\sigma)^2}{\sigma^T\Sigma\sigma}+{8}\,\frac{(\tilde{\sigma}^T\hat{\sigma}-\sigma^T\sigma)^2}{\sigma^T\Sigma\sigma}.
\label{eq:decompoSparse}
\end{equation}

Remark that $\tilde{\lambda}^{-2}=(\tilde{\lambda}^{-1}\lambda)^{2}\lambda^{-2}.$ 
The term $(\tilde{\lambda}^{-1}\lambda)$ represents the ratio between the estimation of the norm of the first PLS component and its true value. 
Using \cref{Prop:LambdalambdaSparse}, 
$$\tilde{\lambda}^{-1}\lambda\leq 112:= F.$$
Consequently, 
\begin{equation}\label{eqn:decomp_sparse}
	\frac{2}{n}\|X(\bar{\beta}-\hat\beta_{sPLS})\|^2\le \underbrace{4F^2 \frac{1}{n}\lambda^{-2}\|X(\tilde{\sigma}-\sigma)\|^2}_{I}+\underbrace{8 F^2 \lambda^{-2}\frac{(\tilde{\sigma}^T\Sigma\tilde{\sigma}-\sigma^T\Sigma\sigma)^2}{\sigma^T\Sigma\sigma}}_{II}+\underbrace{8\frac{(\tilde{\sigma}^T\hat{\sigma}-\sigma^T\sigma)^2}{\sigma^T\Sigma\sigma}}_{III}.
\end{equation}
 The proof naturally falls into three parts.

\subsubsection[Study of term I]{Study of term $I$}

First, we focus on the deviation of $$\frac{1}{n}\|X(\tilde{\sigma}-\sigma)\|^2=(\tilde{\sigma}-\sigma)^T\Sigma(\tilde{\sigma}-\sigma).$$
Let us decompose this quantity as follows, 
\begin{equation}\label{eqn:decompI}
    (\tilde{\sigma}-\sigma)^T\Sigma(\tilde{\sigma}-\sigma)\le 2(\tilde{\sigma}-\sigma)_{\J_{0,1}}^T\Sigma(\tilde{\sigma}-\sigma)_{\J_{0,1}}+2(\tilde{\sigma}-\sigma)_{\J_{0,1}^C}^T\Sigma(\tilde{\sigma}-\sigma)_{\J_{0,1}^C},
\end{equation}
with indexes set $\J_{0,1}$ defined in \cref{lem:support}.
We consider separately the two terms on the right hand side.

Using the inclusion $\J_{0,1}\subset \J_0$,
	\begin{align*}
		\lefteqn{(\tilde{\sigma}-\sigma)_{\J_{0,1}}^T\Sigma(\tilde{\sigma}-\sigma)_{\J_{0,1}}}\\
		&\le2(\hat{\sigma}-\sigma)_{\J_{0,1}}^T\Sigma(\hat{\sigma}-\sigma)_{\J_{0,1}}+2\mu^2S_{\J_{0,1}}^T\Sigma S_{\J_{0,1}}\\
		&=2\big(\hat{\sigma}_{\J_0}-\sigma_{\J_0}-(\hat{\sigma}-\sigma)_{\J_0\cap\J_{0,1}^C}\big)^T\Sigma\big(\hat{\sigma}_{\J_0}-\sigma_{\J_0}-(\hat{\sigma}-\sigma)_{\J_0\cap\J_{0,1}^C}\big)+2\mu^2S_{\J_{0,1}}^T\Sigma S_{\J_{0,1}}\\
		&\le 4(\hat{\sigma}-\sigma)_{\J_0}^T\Sigma_{\J_0}(\hat{\sigma}-\sigma)_{\J_0}+4(\hat{\sigma}-\sigma)^T_{\J_0\cap\J_{0,1}^C}\Sigma(\hat{\sigma}-\sigma)_{\J_0\cap\J_{0,1}^C}+2\mu^2\rho(\Sigma_{\J_0})|\J_0|.
	\end{align*}
As stated in \cref{lem:support}, on $\M_\delta$, $\J_{0,1}^C\subset\hat{\mathcal{J}}^C$. Hence, writing $\1=(1,\dots, 1)^T \in \mathbb{R}^p$,
$$
(\hat{\sigma}-\sigma)^T_{\J_0\cap\J_{0,1}^C}\Sigma(\hat{\sigma}-\sigma)_{\J_0\cap\J_{0,1}^C} \le (\mu\1+\frac{\mu}{2}\1)^T_{\J_0\cap\J_{0,1}^C}\Sigma_{\J_0}(\1\mu+\1\frac{\mu}{2})_{\J_0\cap\J_{0,1}^C}
\leq \frac{9\mu^2}{4}\rho(\Sigma_{\J_{0}})|\J_0|,
$$
since for all $j\in\J_{0,1}^C$, $|\sigma_j|<\mu/2$. Consequently,
\begin{equation} \label{eq:temp1}
    (\tilde{\sigma}-\sigma)_{\J_{0,1}}^T\Sigma(\tilde{\sigma}-\sigma)_{\J_{0,1}}\le 4(\hat{\sigma}-\sigma)_{\J_0}^T\Sigma_{\J_0}(\hat{\sigma}-\sigma)_{\J_0}+{11}\mu^2\rho(\Sigma_{\J_0})|\J_0|.
\end{equation}

Similarly,
	\begin{align}
		\nonumber (\tilde{\sigma}-\sigma)_{\J_{0,1}^C}^T\Sigma(\tilde{\sigma}-\sigma)_{\J_{0,1}^C}		&\le\sigma_{\J_{0,1}^C}^T\Sigma\sigma_{\J_{0,1}^C}\\
		\nonumber &=\sigma_{\J_0\cap\J_{0,1}^C}^T\Sigma_{\J_0}\sigma_{\J_0\cap\J_{0,1}^C}\\
		\label{eq:temp2} &\le2\frac{\mu^2}{4}\rho(\Sigma_{\J_{0}})|\J_0|.
	\end{align}

With inequalities \eqref{eqn:decompI}, \eqref{eq:temp1} and \eqref{eq:temp2}, we obtain
\begin{equation*}
    (\tilde{\sigma}-\sigma)^T\Sigma(\tilde{\sigma}-\sigma)\leq 8(\hat{\sigma}-\sigma)_{\J_0}^T\Sigma_{\J_0}(\hat{\sigma}-\sigma)_{\J_0}+23\mu^2\rho(\Sigma_{\J_0})|\J_0|.
\end{equation*}
We bound the first term using inequality \eqref{eq13Sparse}, since we are working on the event $\B_\delta$. For the second term, we replace $\mu$ by its expression. We obtain
\begin{equation}\label{Eq:Term1ln}
\frac{(\tilde{\sigma}-\sigma)^T\Sigma(\tilde{\sigma}-\sigma)}{\lambda^2}\le 8\,g(x_{0,\delta})\frac{\tau^2}{n}\frac{\Tr(\Sigma_{\J_0}^2)}{\lambda^2}+8\cdot23\,\frac{\tau^2}{n}\ln\left(\frac{2p}{\delta}\right)\frac{\rho(\Sigma_{\J_0})|\J_0|}{\lambda^2}.
\end{equation}
Hence, for $0<\delta<1/2$,
\begin{equation}\label{eq:termI}
\frac{(\tilde{\sigma}-\sigma)^T\Sigma(\tilde{\sigma}-\sigma)}{\lambda^2}\le C_{I}\,\ln\left(\frac{p}{\delta}\right)\,\frac{\tau^2}{n}\,\frac{\rho(\Sigma_{\J_0})\Tr(\Sigma_{\J_0})}{\lambda^2},
\end{equation}
with $C_{I}$ a positive constant.

 \subsubsection[Study of term II]{Study of term $II$}

The task is now to bound the term  $$\lambda^{-2} \frac{(\tilde{\sigma}^T\Sigma\tilde{\sigma}-\sigma^T\Sigma\sigma)^2}{\sigma^T\Sigma\sigma}.$$
We use the equality $$\tilde{\sigma}^T\Sigma\tilde{\sigma}-\sigma^T\Sigma\sigma=(\tilde{\sigma}-\sigma)^T\Sigma(\tilde{\sigma}-\sigma)+2\sigma^T\Sigma(\tilde{\sigma}-\sigma).$$ 
Using the Cauchy Schwarz inequality on the last term, $$\tilde{\sigma}^T\Sigma\tilde{\sigma}-\sigma^T\Sigma\sigma\le (\tilde{\sigma}-\sigma)^T\Sigma(\tilde{\sigma}-\sigma)+2\sqrt{\sigma^T\Sigma\sigma}\sqrt{(\tilde{\sigma}-\sigma)^T\Sigma(\tilde{\sigma}-\sigma)}.$$
	Then, 
	$$(\tilde{\sigma}^T\Sigma\tilde{\sigma}-\sigma^T\Sigma\sigma)^2\le 2\big((\tilde{\sigma}-\sigma)^T\Sigma(\tilde{\sigma}-\sigma)\big)^2+8\sigma^T\Sigma\sigma\times (\tilde{\sigma}-\sigma)^T\Sigma(\tilde{\sigma}-\sigma).$$
Using inequality \eqref{Eq:Term1ln} on the first term, we get
$$
{\frac{\bigl((\tilde{\sigma}-\sigma)^T\Sigma(\tilde{\sigma}-\sigma)\bigr)^2}{\sigma^T\Sigma\sigma\,\lambda^{2}}}
		\le\bigg(2\cdot 64 \cdot g(x_{0,\delta})^2+2\cdot224^2\bigg)\big(\frac{\tau^2}{n}\big)^2\frac{\rho(\Sigma_{\J_0})^2|\J_0|^2}{\sigma^T\Sigma\sigma\lambda^2}\ln\left(\frac{2p}{\delta}\right)^2.
$$
Applying inequality \eqref{eq:termI} on the second term, we get
$$
		{\frac{\sigma^T\Sigma\sigma(\tilde{\sigma}-\sigma)^T\Sigma(\tilde{\sigma}-\sigma)}{\sigma^T\Sigma\sigma\,\lambda^{2}}}
		\le C_{I}\,\ln\left(\frac{p}{\delta}\right)\,\frac{\tau^2}{n}\,\frac{\rho(\Sigma_{\J_0})\Tr(\Sigma_{\J_0})}{\lambda^2}.
$$
We bound $1/(\sigma^T\Sigma\sigma)$ in the first term using \ref{ass:A2}. We obtain
\begin{multline*}
    {\frac{(\tilde{\sigma}^T\Sigma\tilde{\sigma}-\sigma^T\Sigma\sigma)^2}{\sigma^T\Sigma\sigma\,\lambda^2}}
		\le 2(128\,g(x_{0,\delta})^2+2\cdot224^2)d_{\delta,p}^{-1}\ln\left(\frac{2p}{\delta}\right)^2\,\frac{\tau^2}{n}\,\frac{\rho(\Sigma_{\J_0}) \Tr(\Sigma_{\J_0})}{\,\lambda^2}\\
		+8C_{I}\,\ln\bigl(\frac{p}{\delta}\bigr)\,\frac{\tau^2}{n}\,\frac{\rho(\Sigma_{\J_0})\Tr(\Sigma_{\J_0})}{\lambda^2}.
\end{multline*}
When $d_{\delta,p}= 4\,g(x_{0,\delta})+192 \ln(\frac{2p}{\delta}),$ we deduce
\begin{equation} \label{eq:termII}
    \lambda^{-2}{\frac{(\tilde{\sigma}^T\Sigma\tilde{\sigma}-\sigma^T\Sigma\sigma)^2}{\sigma^T\Sigma\sigma}}
		\le C_{II,\delta}\,\ln\left(\frac{p}{\delta}\right)\,\frac{\tau^2}{n}\,\frac{\rho(\Sigma_{\J_0})\Tr(\Sigma_{\J_0})}{\lambda^2},
\end{equation}
where $C_{II}$ is a positive constant.

\subsubsection[Study of term III]{Study of term $III$}

Finally we focus on the last term,
$$\frac{(\tilde{\sigma}^T\hat{\sigma}-\sigma^T\sigma)^2}{\sigma^T\Sigma\sigma}.$$
Using the inclusion $\hat{\J}\subset\J_0$ given by \cref{lem:support}, we get
\begin{align*}
	\tilde{\sigma}^T\hat{\sigma}-\sigma^T\sigma&=\tilde{\sigma}_{\J_0}^T\hat{\sigma}_{\J_0}-\sigma_{\J_0}^T\sigma_{\J_0}\\
	&=(\tilde{\sigma}-\sigma)^T_{\J_0}(\hat{\sigma}-\sigma)_{\J_0}+\sigma_{\J_0}^T(\hat{\sigma}-\sigma)_{\J_0}+\sigma_{\J_0}^T(\tilde{\sigma}-\sigma)_{\J_0}\\
	&\le \sqrt{(\tilde{\sigma}-\sigma)^T(\tilde{\sigma}-\sigma)}\sqrt{(\hat{\sigma}-\sigma)_{\J_0}^T(\hat{\sigma}-\sigma)_{\J_0}}\\
	&\qquad +\sqrt{\sigma^T\sigma}\bigg(\sqrt{(\tilde{\sigma}-\sigma)^T(\tilde{\sigma}-\sigma)}+\sqrt{(\hat{\sigma}-\sigma)_{\J_0}^T(\hat{\sigma}-\sigma)_{\J_0}}\bigg).
\end{align*}
Using the inequality $2ab\le a^2+b^2$ for all $a,b\in \mathbb{R}$, it follows that
\begin{align}
\nonumber \lefteqn{(\tilde{\sigma}^T\hat{\sigma}-\sigma^T\sigma)^2}\\
\nonumber &\le 2(\tilde{\sigma}-\sigma)^T(\tilde{\sigma}-\sigma)\times (\hat{\sigma}-\sigma)_{\J_0}^T(\hat{\sigma}-\sigma)_{\J_0}+4\,\sigma^T\sigma\times \big((\tilde{\sigma}-\sigma)^T(\tilde{\sigma}-\sigma)+(\hat{\sigma}-\sigma)_{\J_0}^T(\hat{\sigma}-\sigma)_{\J_0}\big)\\
\label{eq:decompIII}  &\le \big((\tilde{\sigma}-\sigma)^T(\tilde{\sigma}-\sigma)\big)^2+\big((\hat{\sigma}-\sigma)_{\J_0}^T(\hat{\sigma}-\sigma)_{\J_0}\big)^2+4\,\sigma^T\sigma\big((\tilde{\sigma}-\sigma)^T(\tilde{\sigma}-\sigma)+(\hat{\sigma}-\sigma)_{\J_0}^T(\hat{\sigma}-\sigma)_{\J_0}\big).
\end{align}

First, we concentrate our attention on the quantity $(\tilde{\sigma}-\sigma)^T(\tilde{\sigma}-\sigma)$. Using \cref{lem:support}, we obtain
\begin{align*}
		(\tilde{\sigma}-\sigma)^T(\tilde{\sigma}-\sigma)&\le 2(\tilde{\sigma}-\sigma)_{\J_{0,1}}^T(\tilde{\sigma}-\sigma)_{\J_{0,1}}+2(\tilde{\sigma}-\sigma)^T_{\J_{0,1}^C}(\tilde{\sigma}-\sigma)_{\J_{0,1}^C}\\
		&\le2(\hat{\sigma}-\sigma-\mu S)_{\J_{0,1}}^T(\hat{\sigma}-\sigma-\mu S)_{\J_{0,1}}+2\sigma_{\J_0\cap\J_{0,1}^C}^T\sigma_{\J_0\cap\J_{0,1}^C}\\
		&\le 4(\hat{\sigma}-\sigma)_{\J_0}^T(\hat{\sigma}-\sigma)_{\J_0}+4\mu^2S_{\J_0}^TS_{\J_0}+\frac{1}{2}\mu^2|\J_0|\\
		&\le 4(\hat{\sigma}-\sigma)_{\J_0}^T(\hat{\sigma}-\sigma)_{\J_0}+\frac{9}{2}\mu^2|\J_0|.
\end{align*}
Starting from \eqref{eq:decompIII}, we obtain 
\begin{align*}
(\tilde\sigma^T\hat\sigma - \sigma^T\sigma)^2 
& \leq \left( 4(\hat\sigma-\sigma)^T_{\mathcal{J}_0}(\hat\sigma-\sigma)_{\mathcal{J}_0} + \frac{9}{2} \mu^2 | \mathcal{J}_0|\right)^2 + \left( (\hat\sigma-\sigma)^T_{\mathcal{J}_0}(\hat\sigma-\sigma)_{\mathcal{J}_0}\right)^2 \\
& \hspace{2cm} + 4\sigma^T\sigma \left( 5 (\hat\sigma-\sigma)^T_{\mathcal{J}_0}(\hat\sigma-\sigma)_{\mathcal{J}_0} + \frac{9}{2} \mu^2 |\mathcal{J}_0| \right). 
\end{align*}

Note that, on the event $\B_\delta$, equation \eqref{eq12Sparse} in \cref{Prop:Majoration3TermesSparse} reads as
\begin{equation*}\label{eq:norm_sigma_hat_sparse}
(\hat{\sigma}-\sigma)_{\J_0}^T(\hat{\sigma}-\sigma)_{\J_0}\le g(x_{0,\delta})\frac{\tau^2}{n}\Tr(\Sigma_{\J_0}).
\end{equation*}
Hence,
\begin{multline*}
{(\tilde{\sigma}^T\hat{\sigma}-\sigma^T\sigma)^2}
\le 33\,g(x_{0,\delta})^2(\frac{\tau^2}{n})^2\Tr(\Sigma_{\J_0})^2+2\cdot36^2(\frac{\tau^2}{n})^2\ln\bigl(\frac{2p}{\delta}\bigr)^2|\J_0|^2 \\ + 4\,\sigma^T\sigma\bigg(5g(x_{0,\delta})\frac{\tau^2}{n}\Tr(\Sigma_{\J_0})+36\,\frac{\tau^2}{n}\ln(\frac{2p}{\delta})|\J_0|\bigg).
\end{multline*}
Observe that \ref{ass:A2} yields 
$$\sigma^T\Sigma\sigma=\lambda\sigma^T\sigma\geq\lambda \frac{\tau^2}{n}\,\Tr(\Sigma_{\J_0}).$$  
Hence,
\begin{align}
 \nonumber   \frac{(\tilde{\sigma}^T\hat{\sigma}-\sigma^T\sigma)^2}{\sigma^T\Sigma\sigma}    
    & \le 33\,\frac{g(x_{0,\delta})^2}{d_{\delta,p}} \frac{\tau^2}{n}\frac{\Tr(\Sigma_{\J_0})}{\lambda}
    +{2\cdot36^2}\frac{\ln\left(\frac{2p}{\delta}\right)^2}{d_{\delta,p}}\frac{\tau^2}{n}\frac{\Tr(\Sigma_{\J_0})}{\lambda}\\
\nonumber    & \quad + 20\,g(x_{0,\delta})\frac{\tau^2}{n}\frac{\Tr(\Sigma_{\J_0})}{\lambda}+144\,\frac{\tau^2}{n}\ln\left(\frac{2p}{\delta}\right)\frac{\Tr(\Sigma_{\J_0})}{\lambda}\\
 \label{eq:termIII}   & \le  C_{III,\delta}\ln(\frac{p}{\delta})\frac{\tau^2}{n}\frac{\Tr(\Sigma_{\J_0})}{\lambda}
\end{align}
where $C_{III,\delta}$ is a positive constant depending only on $\delta$, when $d_{\delta,p}= 4\,g(x_{0,\delta})+192 \ln(\frac{2p}{\delta})$.

\subsubsection{End of the proof}

\label{s:ddelta}

\cref{Th:M2} follows directly from \eqref{eqn:decomp_sparse}, \eqref{eq:termI}, \eqref{eq:termII} and \eqref{eq:termIII}. It has been proven under \ref{ass:A2} when $d_{\delta,p}=4g(x_{0,\delta})+192\ln\bigl(\frac{2p}{\delta}\bigr)$. Hence, the result still holds considering $d_{\delta,p}=C_0\left(\ln\bigl(\frac{10}{\delta}\bigr)+\ln\bigl(\frac{p}{\delta}\bigr)\right)$ with $C_0=384$.

\subsection[Proof of Corollary 4.2]{Proof of \cref{Cor:M2}}
\label{s:proof_cor}

Since the support of $\sigma$ is $\J_0$, we have immediately that 
$$ \|\sigma_{\J_0^c}\|_1 = 0 < \| \sigma_{\J_0}\|_1.$$
Hence, using \ref{ass:A3}, we get that
$$ \lambda:= \frac{\sigma^T\Sigma\sigma}{\sigma^T\sigma} = \frac{1}{n} \frac{\| X\sigma\|^2}{\|\sigma\|^2} > \frac{1}{\phi}.$$
A direct application of \cref{Th:M2} then indicates that 
\begin{align*}
\frac{1}{n}\|X(\hat\beta_{sPLS}-\beta)\|^2
& \le  \frac{2}{n}\underset{v\in[\sigma]}{\inf}\|X(\beta-v)\|^2 +  D_{\delta}\frac{\tau^2 s}{n}\max\left(\frac{\rho(\Sigma_{\J_0})}{\lambda^2},\frac{1}{\lambda} \right) \ln\left(\frac{p}{\delta}\right)\\
& \leq \frac{2}{n}\underset{v\in[\sigma]}{\inf}\|X(\beta-v)\|^2 +  D_{\delta}\frac{\tau^2 s}{n}\max\left(\phi^2\rho(\Sigma_{\J_0}),\phi \right) \ln\left(\frac{p}{\delta}\right), 
\end{align*}
which exactly corresponds to the desired results.

\subsection[Proof of Theorem 4.2]{Proof of \cref{Th:M3}}
\label{s:proof3}

We first remark that the estimator \eqref{eq:bsparse} can be rewritten as 
$$ \tilde\beta = \frac{\tilde \sigma^T \tilde\sigma}{\tilde \sigma^T \Sigma \tilde \sigma}\tilde \sigma:= (\tilde \lambda^\star)^{-1} \tilde \sigma \quad \mathrm{where} \quad \tilde\lambda^\star = \frac{\tilde \sigma^T \Sigma \tilde \sigma}{\tilde \sigma^T \tilde\sigma}.$$

We use similar steps as those used in \cref{Th:M2}. First
$$\frac{1}{n}\|X\tilde{\beta}-X\beta\|^2\le\frac{2}{n}\|X(\tilde{\beta}-\bar{\beta})\|^2+\frac{2}{n}\|X(\bar{\beta}-\beta)\|^2 \quad \mathrm{with} \quad \bar{\beta}=\frac{\sigma^T\sigma}{\sigma^T\Sigma\sigma}\sigma.$$
Then, using the previous segmentation studied in the proofs of \cref{Th:M1} and \cref{Th:M2}, we get
$$\tilde{\beta}-\bar{\beta}=(\tilde{\lambda}^\star)^{-1}\tilde{\sigma}-\lambda^{-1}\sigma=(\tilde{\lambda}^\star)^{-1}(\tilde{\sigma}-\sigma)-(\tilde{\lambda}^\star)^{-1}\frac{\tilde{\sigma}^T\Sigma\tilde{\sigma}-\sigma^T\Sigma\sigma}{\sigma^T\Sigma\sigma}\sigma+\frac{\tilde{\sigma}^T\tilde{\sigma}-\sigma^T\sigma}{\sigma^T\Sigma\sigma}\sigma.$$

It yields
\begin{equation}\frac{2}{n}\|X(\bar{\beta}-\tilde{\beta})\|^2
\le \frac{4}{n}(\tilde{\lambda}^\star)^{-2}\|X(\tilde{\sigma}-\sigma)\|^2+{8}\,(\tilde{\lambda}^\star)^{-2}\frac{(\tilde{\sigma}^T\Sigma\tilde{\sigma}-\sigma^T\Sigma\sigma)^2}{\sigma^T\Sigma\sigma}+{8}\,\frac{(\tilde{\sigma}^T\tilde{\sigma}-\sigma^T\sigma)^2}{\sigma^T\Sigma\sigma}.
\label{eq:decompoSparse2}
\end{equation}
From now on, we work on the event $\mathcal{M}_\delta \cap \mathcal{B}_\delta$. We first concentrate our attention on $\tilde\lambda^\star$. Using \cref{lem:support}, $\hat{\mathcal{J}} \subset \mathcal{J}_0$. Hence $\| \tilde \sigma_{\mathcal{J}_0^c} \|_1 = 0 \leq 3 \| \tilde \sigma_{\mathcal{J}_0} \|_1$, which entails, using \ref{ass:A3} that
$$ \tilde \lambda^\star:= \frac{\tilde \sigma^T \Sigma \tilde\sigma}{\tilde\sigma^T \tilde \sigma} = \frac{1}{n} \frac{\| X\tilde\sigma\|^2}{\|\tilde\sigma\|_2^2} \geq \frac{1}{\phi}.$$
Then inequality \eqref{eq:decompoSparse2} immediately leads to
\begin{equation}\frac{2}{n}\|X(\bar{\beta}-\tilde{\beta})\|^2
\le \frac{4}{n}\phi^2 \|X(\tilde{\sigma}-\sigma)\|^2+{8}\,\phi^2\frac{(\tilde{\sigma}^T\Sigma\tilde{\sigma}-\sigma^T\Sigma\sigma)^2}{\sigma^T\Sigma\sigma}+{8}\,\frac{(\tilde{\sigma}^T\tilde{\sigma}-\sigma^T\sigma)^2}{\sigma^T\Sigma\sigma}.
\label{eq:decompoSparse3}
\end{equation}
A direct application of \cref{Prop:Majoration3TermesSparse} then allows to control each of the terms involved in \eqref{eq:decompoSparse3} and then to conclude the proof.


\printbibliography

\end{document}